\documentclass[12pt,oneside]{article}
\usepackage{amsmath, amssymb, amsthm, amscd, amsfonts, graphicx, fancyhdr, color, mathrsfs}
\input epsf.tex
\usepackage{graphicx}
\usepackage{amssymb}
\usepackage{amsthm}
\usepackage{latexsym}
\usepackage{amsmath}
\usepackage{lscape}
\usepackage{palatino, url, multicol, hyperref}
\usepackage{subfig}
\usepackage{multirow}
\usepackage{geometry}                
\theoremstyle{theorem}
\newtheorem{theorem}{Theorem}[section]
\newtheorem{lemma}[theorem]{Lemma}

\theoremstyle{definition}
\newtheorem{definition}[theorem]{Definition}

\newtheorem{remark}[theorem]{Remark}
\newtheorem{corollary}[theorem]{Corollary}
\newtheorem{example}[theorem]{Example}

\newcommand{\bm}{Bouw-M\"oller }
\newcommand{\ttm}[4]{\begin{bmatrix} #1 & #2 \\ #3 & #4 \end{bmatrix}} 
\newcommand{\ts}{\theta_{S}} 
\newcommand{\vm}{V_M} 
\newcommand{\sm}{S_M} 

\title{Cutting Sequences on  Translation Surfaces}
\author{Diana Davis}

\begin{document}
\maketitle

\pagestyle{myheadings}


\section{Introduction}

We are interested in \emph{translation surfaces}, which are created by identifying opposite congruent edges of polygons.  This creates a flat cone surface, a compact surface having a locally Euclidean metric away from finitely many points; these points are conical singularities whose angles are integer multiples of $2\pi$. Such surfaces often contain distinguished graphs: The surfaces are built by identifying the edges of a union of polygons, and the edges of the polygons form the graph.

The simplest example of a translation surface is the flat torus, where the left and right edges of a square are identified, and the top and bottom edges are identified, in the usual way. The torus has been extensively studied and is well understood \cite{Lothaire}, \cite{Series}. Recently, John Smillie and Corinna Ulcigrai studied the regular octagon surface, where the four pairs of parallel edges are identified \cite{SU}, \cite{SU2}. Subsequent work investigated the double pentagon surface, made by gluing the five pairs of parallel edges of a pair of regular pentagons \cite{DFT}, \cite{Davis}.

%
%

A translation surface always has a group of locally affine automorphisms, and this group has a natural homomorphism to a subgroup of $SL_2 (\mathbf{R})$. The surface is called \emph{Veech} if the image of this homomorphism is a lattice. Even when the surface is not Veech, it has some interesting automorphisms.

We are interested in cutting sequences on translation surfaces.
To generate a cutting sequence, one chooses a direction on a surface, and considers the straight-line flow in that direction. One records the labels of the edges hit by an infinite geodesic, which gives us a bi-infinite sequence of edge labels, called the \emph{cutting sequence}. 
 The purpose of this paper is to describe the effect of the parabolic element of the Veech group on cutting sequences on several  families of translation surfaces.

One question of interest is to characterize all possible cutting sequences for a given translation surface. Applying the parabolic element of the Veech group can help us to do this, by generating new valid cutting sequences from a known cutting sequence. Cutting sequences are also interesting because on the square torus, they give a continued fraction expansion for the slope of the associated trajectory \cite{Series}. John Smillie and Corinna Ulcigrai found a construction for the regular octagon surface that gives something analogous to a continued fraction sequence (\cite{SU}, Theorem 2.3.1).

When the image of an affine automorphism is a parabolic element of $SL_2 (\mathbf{R})$, the element is called a \emph{shear}. 
We consider the composition of a vertical reflection with this shear, an action that we call the \emph{flip-shear}. Our main result characterizes the relationship between a cutting sequence corresponding to a trajectory, and the cutting sequence corresponding to the image of that trajectory under the flip-shear. 

%

In general, the shear performs a twist in each cylinder of a surface's cylinder decomposition (Definition \ref{cd}). Each cylinder must be twisted a whole number of times for the shear to be an automorphism, so the magnitude of the twist must be an integer multiple of the modulus of each cylinder. In particular, all the cylinders must have commensurable moduli. 

Square-tiled surfaces are the simplest examples of translation surfaces with commensurable moduli. It turns out (Lemmas \ref{mododd} and \ref{modeven}) that the cylinders of regular polygon surfaces also have commensurable moduli. 
John Smillie and Corinna Ulcigrai gave a rule for the effect of the flip-shear on cutting sequences on the regular octagon surface, and on all regular $2n$-gon surfaces \cite{SU}, \cite{SU2}. Using the same methods, we showed that the same rule  holds for all double regular $n$-gon surfaces for odd $n$ \cite{Davis}.

Irene Bouw and Martin M\"oller \cite{BM} gave algebraic models for a family of translation surfaces, which we now call \emph{\bm surfaces}, whose Veech groups are $(m,n, \infty)$-triangle groups for most $m$ and $n$. This generalized earlier work of William Veech \cite{Veech} and Clayton Ward \cite{Ward}. 

Motivated by the goal of understanding Bouw and M\"oller's work, Pat Hooper \cite{Hooper} and independently Ronen Mukamel defined a doubly-indexed family of translation surfaces by more geometric means. In \cite{Hooper}, Hooper described several inequivalent ways of presenting these surfaces, both in terms of grid graphs and in terms of polygon gluings. In particular, he gave a polygon decomposition for the $(m,n)$-indexed surface consisting of  $m$ polygons that each have $2n$ edges, which we study in $\mathsection$\ref{bmintro}$-\mathsection$\ref{bmresults}. In each of these surfaces, all  the cylinders have the same modulus. Hooper proved that his surfaces are usually the same as those constructed by Bouw and M\"oller, with possible exceptions when $m=n$ or when $m$ and $n$ are both even. Later, Alex Wright \cite{Wright} showed that the Bouw-M\"oller and Hooper surfaces are in fact the same for all $m$ and $n$.
In this paper, we give a rule (Theorem \ref{bmthm}) for the effect of the flip-shear on cutting sequences on \bm surfaces.

\subsection*{Results presented in this paper}

We characterize the effect of the
flip-shear
  on all \emph{perfect translation surfaces} with \emph{common modulus $M$} (Definition \ref{pts}). These include many families of surfaces, such as:
regular polygon surfaces,
\bm surfaces, and
many rectilinear surfaces.

The results for regular polygon surfaces were already known \cite{Davis}, \cite{SU}, but the results for \bm surfaces and rectilinear surfaces are new. Roughly speaking, given a cutting sequence corresponding to a geodesic trajectory on a perfect translation surface, and the new cutting sequence corresponding to image of that trajectory under the flip-shear, to get from the original sequence to the new sequence, we keep the sandwiched letters, and also keep certain other letters. The kept letters are sandwiched in about half of the cases (see Table $\ref{34kept}$).

In  $\mathsection$\ref{definitions}, we  give the general definitions. In $\mathsection$\ref{arcs}, we discuss the action of the flip-shear in detail. In $\mathsection$\ref{mainresult}, we  prove the main result. In the remainder of the paper, we discuss how the result applies to the regular polygon surfaces ($\mathsection$\ref{regularpolygons}) and \bm surfaces ($\mathsection$\ref{bmintro}$-\mathsection$\ref{bmresults}). For a discussion of rectilinear surfaces, including certain L-shaped tables, see (\cite{thesis}, Chapter 8).

\subsection*{Acknowledgements}
Many thanks to my advisor, Richard Schwartz, generally for his excellent guidance throughout my time in graduate school, and specifically for the multitude of insightful suggestions he made for improving this paper (a subset of my thesis). Corinna Ulcigrai, John Smillie and I spent an enlightening three days in Bristol, UK in February 2012. We had extensive discussions about many aspects of the \bm project, and I benefited from their insights. Sergei Tabachnikov and Pat Hooper shared their questions and insights with me over the course of several years. Ronen Mukamel taught me about cylinder decompositions and moduli.

\section{Definitions}
 \label{definitions} 
 \subsection{Translation surfaces}

\begin{definition} \label{cd}
Let $S$ be a translation surface. A \emph{cylinder decomposition} of $S$ is a partition of $S$ into finitely many flat parallel cylinders. (A \emph{flat} cylinder does not contain a singular point in its interior.) The interiors of these cylinders are pairwise disjoint, and their union is $S$.  Figure \ref{octcyl} shows a cylinder decomposition for the regular octagon surface.
\end{definition}

\begin{figure}[!h] 
\centering
\includegraphics[width=300pt]{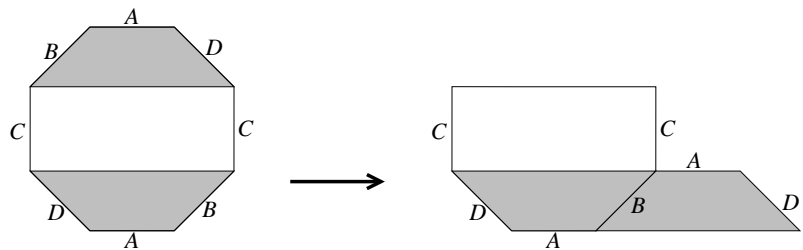}
\begin{quote}\caption{A cylinder decomposition for the regular octagon surface \label{octcyl}} \end{quote}
\vspace{-1em}
\end{figure}

\begin{definition} \label{leveldef}
A polygon is \emph{level} if it never happens that there is a horizontal line that both contains  a vertex of the polygon and crosses an edge of $P$ at an interior point. 
\end{definition}
%

\begin{definition}
A polygon is \emph{vertically symmetric} if it has a line of vertical symmetry. This means that we may translate the polygon in the plane so that the map $(x,y)\mapsto (-x,y)$ is a symmetry of the polygon. 
\end{definition}

\begin{definition} \label{specialdef}
A polygon is \emph{special} if it is convex, level and vertically symmetric. 
%
A translation surface has a \emph{special decomposition} if it can be obtained by gluing together pairs of parallel edges of a finite union of special polygons. We require that the gluing is compatible with the vertical reflection symmetries of the polygons: Suppose that  edge $e_1$ of polygon $P_1$ is glued to edge $e_2$ of polygon $P_2$. Let $I_1$ and $I_2$, respectively, be the vertical reflections preserving $P_1$ and $P_2$. Then $I_1 (e_1)$ is glued to $I_2 (e_2)$. 

If a given translation surface has a special decomposition, we call it a \emph{special translation surface}. Whenever we deal with a special translation surface, we will assume that it is given to us by way of a special decomposition.
\end{definition}

\begin{figure}[!h] 
\centering
\includegraphics[height=90pt]{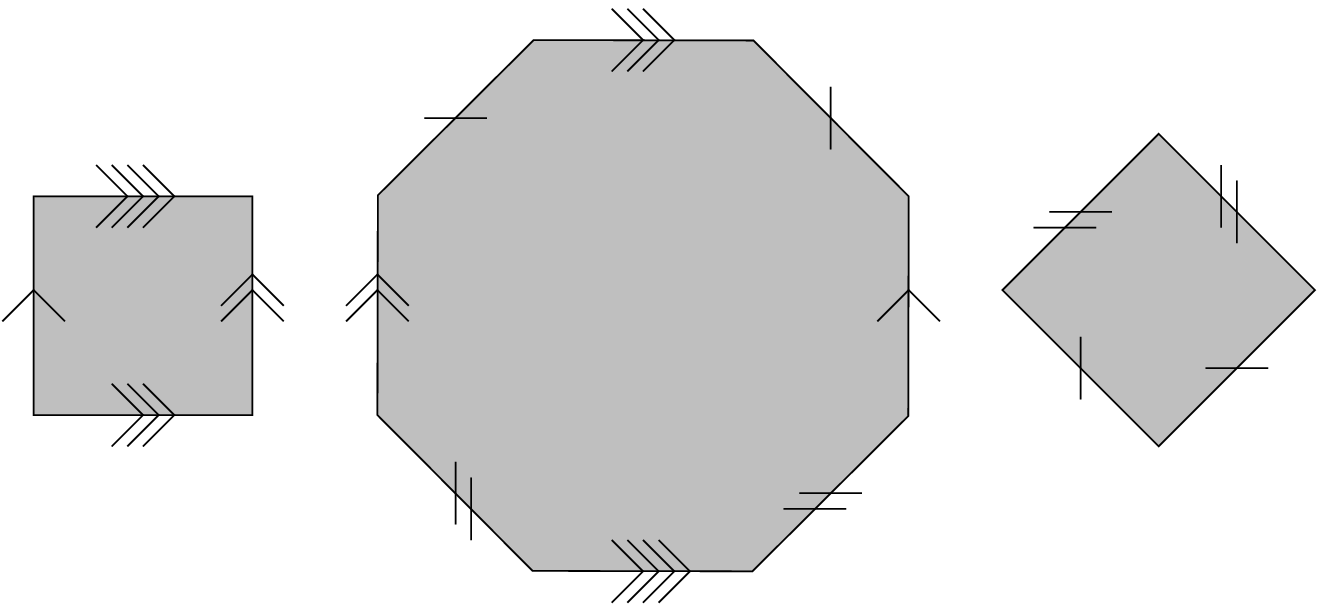} \hspace{3em}
\includegraphics[height=108pt]{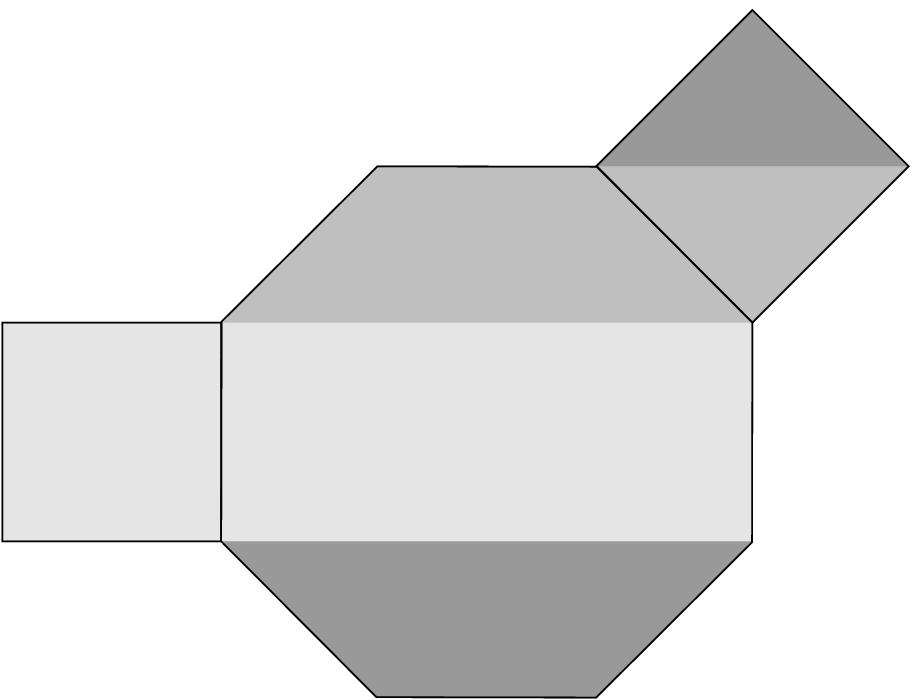}
\begin{quote}\caption{(a) A special decomposition, and (b) the horizontal cylinder decomposition, of the $(3,4)$ \bm \mbox{surface} \label{specdecomp}\label{hcd}} \end{quote}
\end{figure}

Figure \ref{specdecomp} (a) shows a special decomposition of the $(3,4)$ \bm surface. It turns out that all  \bm surfaces have special decompositions (Lemma \ref{bmspecial}).

\begin{lemma} \label{hcdexists}
A special surface has a canonical cylinder decomposition, in which all the cylinders are horizontal.
\end{lemma}

\begin{proof}
Consider the union of horizontal line segments connecting vertices of polygons in the special decomposition. These line segments piece together to give compact $1$-manifolds without boundary, and hence circles. These circles are horizontal on the surface, and divide it into cylinders. This is clearly a cylinder partition.
\end{proof}

Figure \ref{hcd} (b) shows how the special decomposition in Figure \ref{specdecomp} (a) gives rise to a horizontal cylinder decomposition on the surface. Each cylinder is shaded a different color.


\begin{definition}
For us, an \emph{isosceles trapezoid} is a vertically symmetric trapezoid or a vertically symmetric isosceles triangle. We think of the triangle as being a degenerate case of the trapezoid. The three main cases are
generic isosceles trapezoid, an isosceles triangle, and a rectangle.
\end{definition}


\begin{lemma} \label{trappart}
Let $S$ be a special translation surface. Assume that we have equipped $S$ with its canonical cylinder decomposition. Each nonempty intersection of a cylinder with a special polygon is an isosceles trapezoid. Hence, each cylinder in the cylinder decomposition has a partition into isosceles trapezoids.
\end{lemma}

\begin{proof}
Referring to the proof of Lemma \ref{hcdexists}, Figure \ref{hcd} (b) shows how each region between two horizontal segments is an isosceles trapezoid (a rectangle, isosceles triangle or isosceles trapezoid). The result of the lemma is clear from this fact.
\end{proof}

%


\begin{definition}
A \emph{level automorphism} of a translation surface is a locally affine automorphism that preserves the horizontal direction. 
\end{definition}

Figure \ref{affauto} shows two examples of level automorphisms on the regular octagon surface. One is known as a \emph{flip} and the other is known as a \emph{shear}. The flip has order $2$, and the shear has infinite order. The shear fixes each of the drawn horizontal line segments pointwise; it performs a Dehn twist in each horizontal cylinder. If $R$ is the flip and $S$ is the shear, then the middle picture in Figure \ref{affauto} uses the notation $A'=R(A)$, etc. and the right picture uses the notation $A' = S(A)$, etc.

\begin{figure}[!h] 
\centering
\includegraphics[width=400pt]{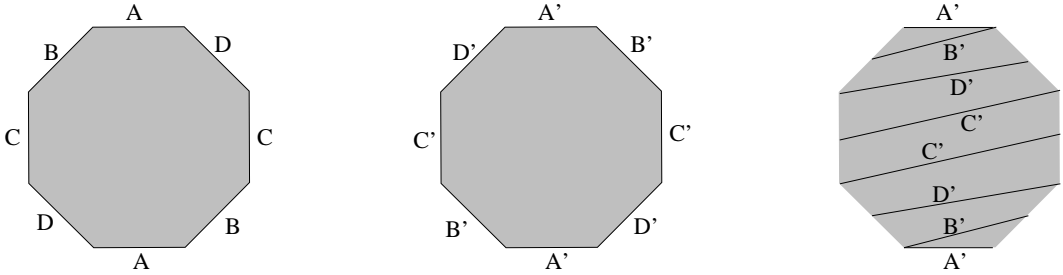}
\begin{quote}\caption{The  octagon surface, and its images under the \emph{flip} and the \emph{shear}  \label{affauto}} \end{quote}
\vspace{-2em}
\end{figure}


\begin{lemma}\label{flipexists}
Let $S$ be a special translation surface. Then $S$ admits an order-$2$ level automorphism.
\end{lemma}

\begin{proof}
One can separately reflect each polygon in the special decomposition through its line of symmetry. The compatibility condition for the gluings guarantees that these separate maps fit together to give a global automorphism of $S$.
\end{proof}

\begin{definition}
Given a special translation surface, we call its order-$2$ affine element from Lemma \ref{flipexists} \emph{the flip} and denote it by $R$.
\end{definition}

We can use the existence of the flip $R$ to give more information about the nature of the trapezoid decompositions of the cylinders in the horizontal cylinder decomposition.

\begin{lemma}
Suppose that $S$ is a special translation surface. Let $\mathscr{C}$ be one of the cylinders in the horizontal cylinder decomposition of $S$. Then $\mathscr{C}$ is partitioned into one or two  trapezoids. In the case where $\mathscr{C}$ is partitioned into one trapezoid, this trapezoid is a rectangle.
\end{lemma}

\begin{proof}
For the first claim, note that the flip $R$ preserves each cylinder in the cylinder decomposition. The restriction $R|_\mathscr{C}$ is an orientation-reversing isometry of $\mathscr{C}$. Any such isometry fixes exactly two vertical line segments in $\mathscr{C}$. These two segments are diametrically opposed. At the same time, by construction, $R$ fixes the vertical segment of symmetry of each trapezoid in the decomposition. Thus, there can be at most two such trapezoids.

For the second claim, suppose that $\mathscr{C}$ is decomposed into a single trapezoid. This means that opposite sides of the trapezoid are glued together. Thus, by definition, they are parallel, so the trapezoid is a rectangle. 
\end{proof}

So far, we have exhibited one level automorphism on a special translation surface. We will exhibit two more. First, we need several more definitions.

\begin{definition}
The \emph{modulus} of a horizontal cylinder is the ratio $\frac W H .$
Here $H$ is the height of the cylinder and $W$ is the length of the horizontal translation needed to identify the opposite edges of a fundamental domain for the cylinder. 
\end{definition}

%

\begin{definition}  \label{casedefs} \label{pcd}
Let $S$ be a special translation surface equipped with its canonical horizontal cylinder decomposition. We call the cylinder decomposition \emph{perfect} if there is some number $M$ such that the following is true of each cylinder $\mathscr{C}$ of the decomposition:
\begin{itemize}
\item $\mathscr{C}$ has modulus $M$ and is decomposed into two isosceles trapezoids; or
\item $\mathscr{C}$ has modulus $M/2$ and is decomposed into a single rectangle.
\end{itemize}
We call $M$ the \emph{common modulus} of the decomposition. If the first case occurs, we call $\mathscr{C}$ \emph{typical}. If the second option occurs, we call $\mathscr{C}$ \emph{exceptional}.
\end{definition}

\begin{remark} It might seem more natural to require that, in the exceptional case, the rectangle have modulus $M$. However, when we take the double cover of the cylinder in this case, we get a cylinder of modulus $M$ that is decomposed into two trapezoids (rectangles). This turns out to be more natural for our purposes.
\end{remark}



\begin{lemma} \label{shearexists}
Let $S$ be a special translation surface with a perfect horizontal cylinder decomposition of common modulus $M$. Then $S$ admits an infinite-order level automorphism that fixes the tops and bottoms of the cylinders pointwise, and has derivative
$$ \ttm 1M01.$$
\end{lemma}

\begin{proof}
In each cylinder, there is a unique affine automorphism that has the desired effect and has the desired derivative. These automorphisms piece together to give the global automorphisms of the surface.
\end{proof}

\begin{definition}
Let $S$ be a special translation surface whose horizontal cylinder decomposition is perfect and has common modulus $M$. We call the infinite-order level automorphism from Lemma \ref{shearexists} \emph{the shear}, and we denote it by $S_M$. 
\end{definition}

\begin{definition} \label{vm}
Let $S$ be a special translation surface whose horizontal cylinder decomposition is perfect and has common modulus $M$. (This guarantees the existence of $R$ and $S_M$.) Then we define
$$V_M = S_M \circ R.$$
We call $V_M$ the \emph{flip-shear}. Corresponding to these elements, we have the composition of derivatives
$$V_M' = S_M' \circ R'.$$
In terms of matrices, this is
$$ \ttm 1{-M}01 = \ttm 1M01 \circ \ttm {-1}001.$$
\end{definition}

\begin{definition} \label{pts}
A \emph{perfect translation surface} is a special translation surface whose cylinder decomposition is perfect. 
\end{definition}

A perfect translation surface admits the level automorphisms $R,S_M$ and $V_M$, where $M$ is the common modulus of the cylinder decomposition. In the next chapter, we will analyze how these elements act on certain distinguished arcs in a perfect translation surface.

\begin{remark} All regular polygon surfaces (with a horizontal edge), all \bm surfaces, and certain square-tiled surfaces are perfect translation surfaces. For regular polygon surfaces, the common modulus is $M=2\cot \pi/n$ (Lemmas \ref{mododd}, \ref{modeven}). For the $(m,n)$ \bm surface, the common modulus is \mbox{$M=2\cot\pi/n +2 \frac{\cos \pi/m}{\sin \pi/n}$} (see Lemma \ref{modbm} or \cite{thesis}, Lemma 6.6). \end{remark}

\subsection{Trajectories and cutting sequences}

The purpose of this paper is to describe the action of $V_M$ on cutting sequences associated to geodesic trajectories on perfect translation surfaces. We analyze $V_M$ because $V_M$, rotations, and reflections generate all the symmetries of the surface. The actions of rotations and reflections are easy to describe, but the action of $V_M$ takes more work. 

Since $V_M = S_M\circ R$, it is also true that $S_M$, rotations and reflections generate the group of symmetries of the surface. So instead of describing the action of $V_M$, we could choose to describe the action of $S_M$. However, the action of $V_M$ turns out to be more elegant than that of $S_M$: The rule for the effect of $S_M$ is essentially ``apply the rule given in this paper for $V_M$, and then permute the edge labels in a certain way.'' So in this paper we give the rule for $V_M$, and from there it is easy to deduce the effect of $S_M$.

If we know something about a certain trajectory, we can apply the symmetries of the surface to learn something about other trajectories on the surface. 
Ideally, we would be able to analyze all the trajectories in some set, and also be able to use the symmetries of the surface to transform all possible trajectories into the set we can analyze, thereby analyzing all possible trajectories. This is possible for regular polygon surfaces, because we can analyze trajectories with $\theta\in[0,\pi/n)$ and then apply the $n$-fold rotational symmetry of the surface to rotate any trajectory into that sector. This is not possible for \bm surfaces, because the $(m,n)$ \bm surface has $n$-fold rotational symmetry, and yet we can only analyze trajectories with $\theta\in[0,\ts)$, where $\ts$ is at most  $\pi/(2n)$. So we cannot analyze trajectories with $\theta\in [\pi/(2n),\pi/n$), which is a limitation of the method. We hope to exploit the surfaces' other symmetries in the future to analyze all possible trajectories.

The main Theorem \ref{mainthm} proves the result for perfect translation surfaces in general, and the rest of the paper applies the result to several families of perfect translation surfaces: Regular polygon surfaces ($\mathsection  \ref{regularpolygons}$) and \bm surfaces ($\mathsection  \ref{bmintro}-\mathsection \ref{bmresults}$). 

We will now define some notation for geodesic trajectories and their associated cutting sequences.

\begin{definition}\label{cs} 
Given a geodesic trajectory on a perfect translation surface, the \emph{cutting sequence} associated to the trajectory is the bi-infinite sequence of polygon edges that the trajectory crosses.  (By convention, if a trajectory hits a vertex, it stops. In that case, the associated cutting sequence is not bi-infinite.) The \emph{derived trajectory} is the image of the trajectory under the flip-shear $V_M$. The \emph{derived sequence} is the cutting sequence associated to the derived trajectory.  
\end{definition}
\begin{lemma} \label{involution}
$ \vm^{-1} =  \vm$, i.e. $\vm$ is an involution.
\end{lemma}

\begin{proof}
By definition, $\vm = S_M \circ R$. So
$$\vm \circ \vm = \sm \circ R \circ  \sm \circ R = \sm \circ (R \circ  \sm \circ R) = \sm \circ (R \circ  \sm \circ R^{-1}).$$
We have $R=R^{-1}$ because $R$ is a reflection, and thus has order $2$. Also, $R \circ \sm R^{-1} = \sm^{-1}$, because a  reflection composed with a perpendicular shear composed with a  reflection is a  shear in the opposite direction. Hence,
$$\vm \circ \vm =\sm \circ (R \circ  \sm \circ R^{-1}) = \sm \circ \sm^{-1} = Id.$$\end{proof}

We use the following notations only in Lemma  \ref{lookatshears}: Let $\tau$ be a geodesic trajectory on a  perfect translation surface with common modulus $M$, and let $\vm \tau$ be the image of $\tau$ under $\vm$, i.e. the derived trajectory. Let $L=\{A,B,\ldots\}$ be the set of polygon edges, and let $L'=\{A',B',\ldots\}$ be their transformed images. Let $c(\tau)$ be the cutting sequence associated to $\tau$, i.e. the sequence of polygon edges that $\tau$ crosses, so $c(\tau) \in L^\mathbb{Z}$. Let $c'(\tau)$ be the sequence of transformed edges that $\tau$ crosses, so $c'(\tau) \in {L'}^\mathbb{Z}$. Let $c(\vm\tau)$ be the cutting sequence associated to $\vm\tau$, i.e. the \emph{derived} sequence. 

\newpage

\begin{lemma} \label{lookatshears}
$c(\vm\tau) = c'(\tau)$.
\end{lemma}
In words: Determining the derived cutting sequence for a given trajectory is equivalent to determining which sheared edges the trajectory crosses.
\begin{proof} 
Consider $ \vm^{-1}\tau$, the transformed trajectory, and the sequence $c( \vm \tau)$ of polygon edges that it crosses. When we apply $\vm$, it acts on both the trajectory and on the polygon edges, so the intersection points are preserved. By Lemma \ref{involution}, the image of $ \vm\tau$ under   $\vm$ is $ \vm( \vm^{-1}\tau) =\tau$, and the images of the polygon edges under $\vm$ are the transformed edges, so the intersection points are then the intersection between the trajectory $\tau$ and the transformed edges. Symbolically, 
$$c(V_M\tau) =c(V_M^{-1}\tau) =  \vm^{-1}\tau \cap L = \vm(\vm^{-1}\tau) \cap \vm(L) = \tau \cap L' = c'(\tau).$$ \end{proof}


Given a trajectory, recall that the derived sequence is the cutting sequence corresponding to the derived trajectory, which is the image of the original trajectory under $\vm$. The main Theorem \ref{mainthm}, and its applications Theorems \ref{regularsandwich} and \ref{bmthm} to various families of surfaces, give a combinatorial rule for determining the derived sequence from the original cutting sequence. It is useful to introduce the following notation to keep track of various types of edge crossings, which use  in the  statement and proof of the main Theorem \ref{mainthm}.

\begin{definition} \label{zeroone}
On a perfect translation surface $S$, a trajectory with $\theta \in [0,\ts)$ cuts through an edge on the left side (which includes the bottom edge) of a polygon and then an edge on the right side (which includes the top edge) of a given polygon. If the left and right edges are on the same level, we call it type $(0)$. If the right edge is a level above, we call it type $(1)$. Each three-letter sequence in the cutting sequence corresponds to crossing two polygons, which give us four cases (see Figure \ref{horizvertbent}): $(00), (10), (01),$ and $(11)$. 
\end{definition}

\begin{figure}[!h] 
\centering
\includegraphics[width=380pt]{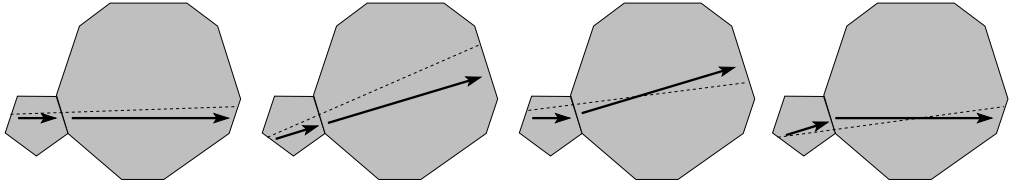}
\begin{quote}\caption{Examples of $(00), (11), (01),$ and $(10)$ cases, with example trajectories (dotted) for each.  \label{horizvertbent}} \end{quote}
\end{figure}

\section{The Action on Edges}
 \label{arcs} \begin{definition}
Let $S$ be a perfect translation surface. A \emph{gluing edge} of $S$ is an arc in $S$ that is the image of a pair of polygon edges under the identification map. Within each cylinder, the gluing edges are  the non-horizontal edges of the trapezoids in the trapezoid partition. Figure \ref{glu} shows a typical cylinder with the gluing edges labeled $A$ and $B$. In an exceptional cylinder that is decomposed into a single rectangle, there is just one gluing edge.
\end{definition}

\begin{figure}[!h] 
\centering
\includegraphics[height=60pt]{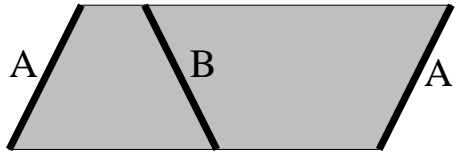}
\begin{quote}\caption{A typical cylinder with its gluing edges $A$ and $B$ \label{glu}} \end{quote}
\end{figure}

\begin{definition} \label{ts}
The \emph{horizontal edges} are the tops and bottoms of the cylinders in the horizontal cylinder decomposition.

The \emph{slanted edges} are the diagonals of positive slope in the trapezoids of the trapezoid decomposition. In terms of the special polygons, each slanted edge joins a vertex on the left side of a polygon to the vertex on the right that is one level above.

In a perfect translation surface $S$, we let $\ts$ be the smallest angle between any slanted edge and the positive horizontal. An example of slanted edges, their associated angles, and $\ts$ is  in Figure \ref{tsdiff}.
\end{definition}

\begin{figure}[!h] 
\centering
\includegraphics[width=200pt]{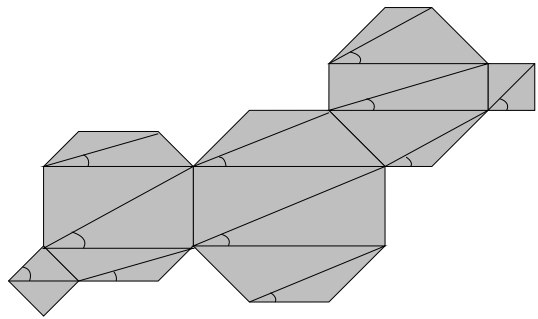}
\begin{quote}\caption{For the $(5,4)$ \bm  surface shown, $\ts$ is the minimum of the marked angles.    \label{tsdiff}} \end{quote}
\end{figure}

 The main Theorem \ref{mainthm}  analyzes a geodesic on a perfect translation surface, where the geodesic's angle is between $0$ and $\ts$.
 
We are interested in the action of $V_M$ on the edges within a cylinder. First, we will describe what happens for a typical cylinder. The degenerate case (when one or more of the trapezoids degenerates to a triangle) is a limiting case of the typical case. The exceptional case (when the cylinder is one rectangle) is a special case of the typical case, using a double cover of the exceptional cylinder.


Figure \ref{unicov} shows a portion of the universal cover of the cylinder $\mathscr{C}$. The universal cover itself is an infinite strip. The cylinder $\mathscr{C}$ is the quotient of the strip by the minimal label-preserving horizontal transformation, which is the generator of the deck transformation group.

\begin{figure}[!h] 
\centering
\includegraphics[width=300pt]{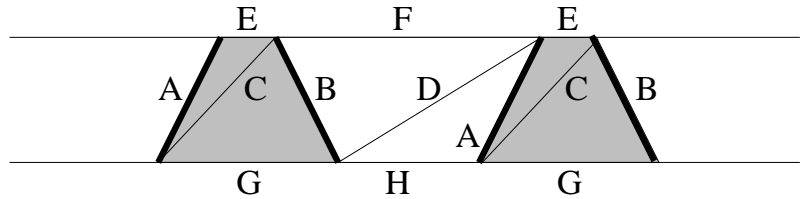}
\begin{quote}\caption{Part of the universal cover of a typical cylinder \label{unicov}} \end{quote}
\end{figure}

We will persistently abuse terminology by saying things like ``The edges labeled $A$ and $B$ in Figure \ref{unicov} are the gluing edges,'' when we really mean that these edges are lifts to the universal cover of the gluing edges.

In Figure \ref{unicov}, $\mathscr{C}$ is decomposed into two trapezoids, one shaded and one unshaded: 
\begin{itemize}
\item Edges $A$ and $B$ are the gluing edges. 
\item Edges $C$ and $D$ are the slanted edges.
\item Edges $E,F,G$ and $H$ are the horizontal edges.
\end{itemize}


Figure \ref{ucflip} (a) shows the action of $R$ on our cylinder. The action of $R$ \emph{does not} lift to an action on the universal cover, so here we are moving the objects in the cylinder by $R$ and then lifting to the universal cover.  In Figure \ref{ucflip}, $a=R(A), b = R(B)$, and so on. 

In the cylinder,  $R$ acts as a symmetry on each of the trapezoids,  swapping the left and right boundaries of each. The effect from Figure \ref{unicov} to Figure \ref{ucflip} (a) on the non-horizontal trapezoid boundaries is that the edge labels $A$ and $B$ have been swapped. $R$ preserves each top and bottom boundary of each trapezoid. The effect from Figure \ref{unicov} to Figure \ref{ucflip} (a) on the top and bottom boundaries is that the edge labels $E,F,G$ and $H$ are preserved.

\newpage
%
Figure \ref{ucshear} (b) shows the action of $V_M$ on $\mathscr{C}$, as depicted in the universal cover. As was the case with $R$, $V_M$ does not lift to an action on the universal cover. Instead, there are two equivalent descriptions of what Figure \ref{ucshear} (b) shows:
\begin{itemize}
\item We let $V_M$ act on  $A,B,C,D,E,F,G,H$  in the cylinder $\mathscr{C}$ (Figure \ref{unicov}), and then lift their images to the universal cover.
\item We lift the action of $S_M$ to the universal cover and let it act on the picture in Figure \ref{ucshear} (a). On the universal cover, $S_M$ is a shear, fixing the bottom boundary and shifting the top boundary to the right by a full modulus distance,  indicated by the horizontal arrow.
\end{itemize}
Since $V_M = R \circ S_M$, these produce the same picture.

\begin{figure}[!h] 
\centering
\includegraphics[width=300pt]{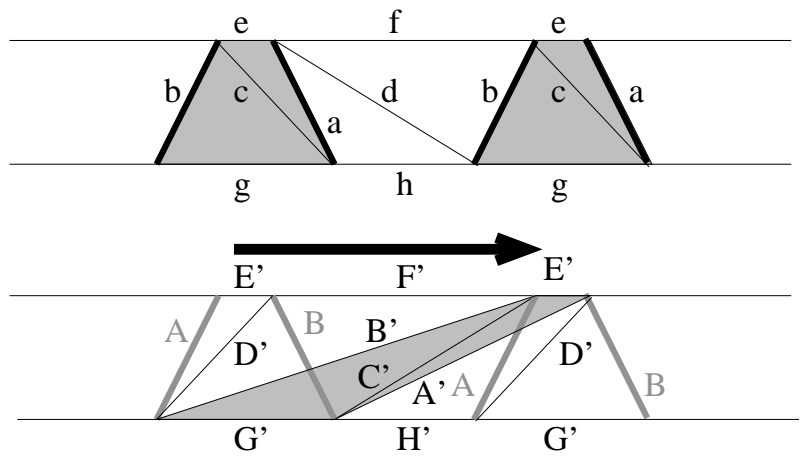}
\begin{quote}\caption{The action of (a) $R$ and (b) $V_M$, depicted in the universal cover \label{ucflip}\label{ucshear}} \end{quote}
\end{figure}


In Figure \ref{ucshear} (b), we have used the notation
$$ A' = S_M (a) = S_M \circ R(A) = V_M (A),$$
and so on.

The original gluing edges $A$ and $B$ are shown in a lighter color. This is important, because the proof of the main Theorem \ref{mainthm} analyzes the configuration of an edge with respect to its image under $V_M$. Here, we can observe that $A'$ intersects $A$, and $B'$ intersects $B$.

The shaded trapezoid in Figure \ref{ucshear} is  (a lift of) $V_M (T)$, where $T$ is the shaded trapezoid in the decomposition of $\mathscr{C}$. This trapezoid is not particularly interesting. More interesting are the \emph{slanted trapezoids}:

\begin{definition} \label{st}
We partition each cylinder into two \emph{slanted trapezoids}, each of which is the region between two slanted edges. 
\end{definition}

Figure \ref{slanted} shows a typical cylinder, depicted in the universal cover. It has two gluing edges ($A$ and $B$), so it has two slanted edges ($C$ and $D$), so it has two slanted trapezoids ($CFDG$ and $DECH$). This is a copy of Figure \ref{unicov}, with different trapezoids shaded. 

\begin{figure}[!h] 
\centering
\includegraphics[width=300pt]{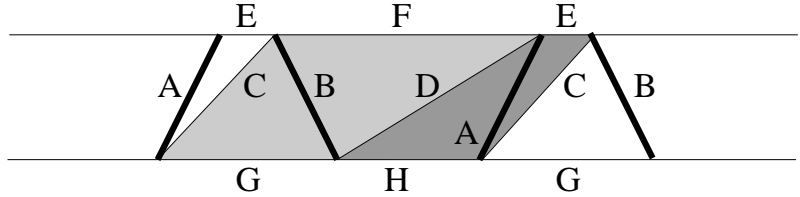} 
\begin{quote}\caption{ A typical cylinder, partitioned into two slanted trapezoids (shaded), depicted in the universal cover    \label{slanted}} \end{quote}
\vspace{-2em}
\end{figure}

We will use slanted trapezoids extensively in proving the main theorem in $\mathsection$\ref{mainresult}. In particular, because we can partition a perfect translation surface into cylinders, and each cylinder into slanted trapezoids, we can partition any perfect translation surface into slanted trapezoids. We will use this partition to partition a geodesic trajectory into finite segments, whose behavior we can easily analyze.

The action is similar in a typical cylinder where one or both of the trapezoids degenerates to a triangle, or where one of the trapezoids degenerates to a slanted edge. 
In an exceptional cylinder, there is only one gluing edge $A$, rather than two gluing edges $A$ and $B$, and the gluing edge is vertical. There is only one slanted edge $C$. We use two copies of the rectangle (one shaded and one unshaded,  as in Figure \ref{ucshear}), which form a double cover of the cylinder, and then the rest of the construction is the same. For expanded details, see (\cite{thesis}, Chapter 3).

%
%
%
%

\section{The Main Result}
 \label{mainresult} 
 
 Recall (Definition \ref{cs}) that the \emph{cutting sequence} associated to a geodesic trajectory on a translation surface is the bi-infinite sequence of polygon edges that the trajectory crosses. The \emph{derived trajectory} is the image of the trajectory under the flip-shear $\vm$, and the \emph{derived sequence} is the cutting sequence associated to the derived trajectory. The purpose of this paper is to determine the relationship between the original cutting sequence and the derived sequence for perfect translation surfaces. Theorem \ref{mainthm} gives the relationship: The derived sequence is a subsequence of the original cutting sequence, obtained by keeping certain letters and removing the rest. It turns out that to determine if a given letter survives into the derived sequence, we only need to look at the letters that precede and follow it, i.e. at the three-letter \emph{word} for which it is the middle letter:

\begin{definition}
Recalling Definition \ref{zeroone}, a \emph{$\mathit{(11)}$ word} is a three-letter sequence corresponding to a trajectory cutting through three edges that form a $(11)$ case. A \emph{$\mathit{(00)}$ word}, a \emph{$\mathit{(01)}$ word} and a \emph{$\mathit{(10)}$ word} are defined analogously.
\end{definition}


\begin{theorem}[Main Theorem] \label{mainthm}
Let $S$ be a perfect translation surface with common modulus $M$.

Consider a geodesic trajectory on the surface whose angle $\theta$ satisfies $0 \leq \theta < \ts$, and its associated cutting sequence. Mark a letter in the  cutting sequence if: 
\begin{itemize}
\item It corresponds to a gluing edge, and it is the middle letter of a $(00)$ or $(11)$ word, or
\item It corresponds to a horizontal edge.
\end{itemize}
Then the derived sequence consists precisely of the marked letters, read from left to right.
\end{theorem}


{\bf Outline of Proof:}  Lemma \ref{lookatshears} shows that determining the derived cutting sequence for a given trajectory is equivalent to determining which sheared edges the trajectory crosses. To determine which sheared edges a trajectory crosses, we partition the trajectory into disjoint segments whose union is the entire trajectory: We partition the translation surface  into finitely many disjoint slanted trapezoids (Definition \ref{st}), and each finite segment of the trajectory that we analyze is the part of the trajectory that lies in a particular slanted trapezoid. Then we analyze what happens on each of those segments, and then pool the information to determine what happens for the entire infinite trajectory. 

To determine the derived sequence from the original cutting sequence, the key question is whether a line segment that passes through edges $E_1 E_2 E_3$ must cross edge $E_2' = \vm(E_2)$, or not. This is a topological question, based on the location of edge $E_2'$ in relation to the edges $E_1, E_2, E_3$. The location of the segment $E_2'$ depends on the configuration of the cylinder for which the edge $E_2$ is a gluing edge: whether it is typical, degenerate or exceptional.  Lemma \ref{zero} and its Corollary \ref{zerocor} are simple topological observations about configurations of line segments in quadrilaterals and trapezoids. Lemmas \ref{hitisgood}-\ref{horiz} apply these observations to each of the disjoint segments of the trajectory. 


The local analysis allows us to locally determine the effect of $\vm$ on each segment, and thus the effect on its associated three-letter word in the cutting sequence, telling us if the middle letter survives into the derived sequence. We then move across the sequence from left to right, looking at the middle letter in each overlapping three-letter word and marking those indicated in Theorem \ref{mainthm}. After we have done this, the derived sequence is  the marked letters, read from left to right.



\begin{lemma} \label{zero} 
Consider a convex quadrilateral with its two diagonals, and a trajectory that intersects the quadrilateral (Figure \ref{quadtrap} (a)). The trajectory crosses both diagonals if and only if the edges it crosses are non-adjacent. In the adjacent case, a diagonal is crossed if and only if it shares the common vertex of the two edges.
\end{lemma}
\nopagebreak
\begin{proof}
Trivial.
\end{proof}

\begin{figure}[!h] 
\centering
\includegraphics[width=210pt]{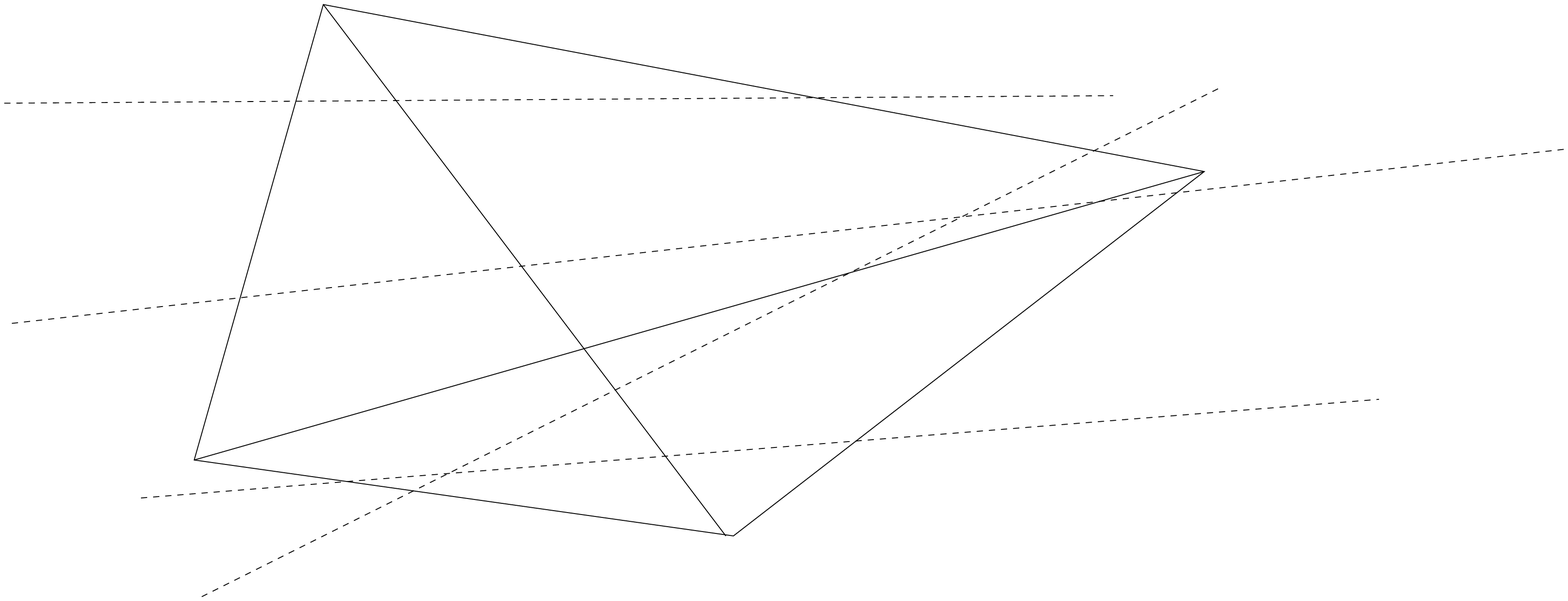} 
\includegraphics[width=210pt]{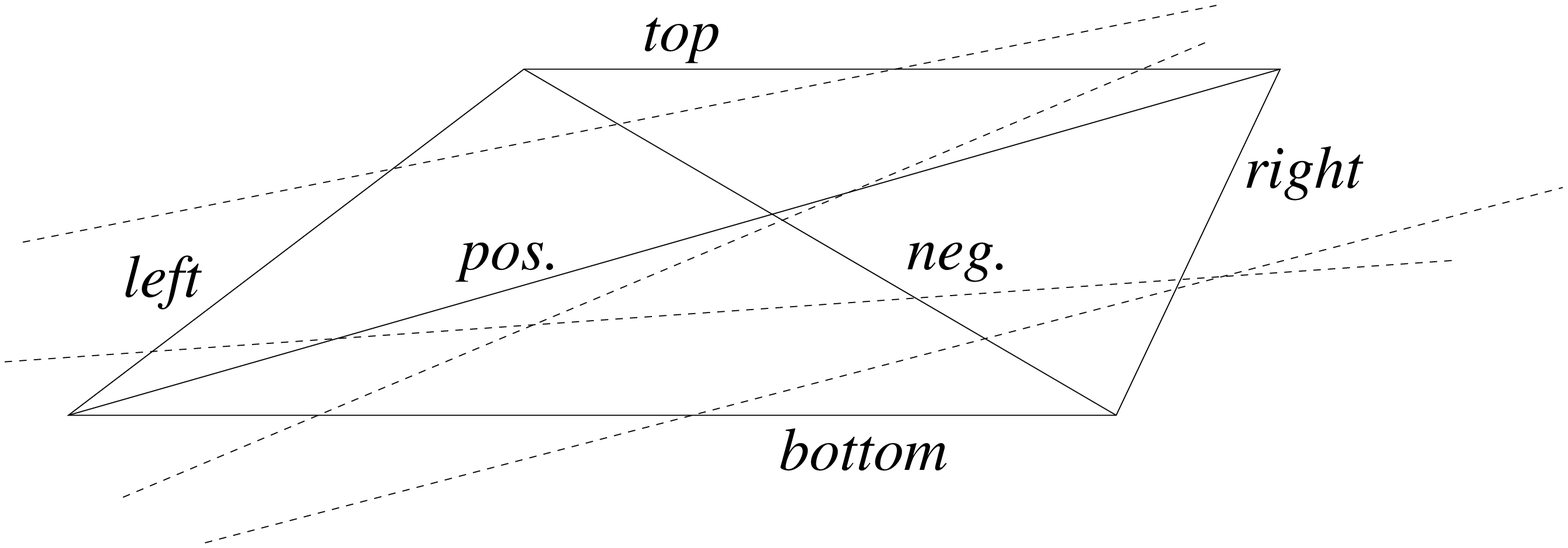} 
\begin{quote}\caption{(a) A convex quadrilateral, and (b) a trapezoid, with trajectories intersecting their diagonals \label{quadtrap}} \end{quote}
\end{figure}


\begin{corollary} \label{zerocor} 
Let the the \emph{top, left, bottom} and \emph{right} edges be as in Figure \ref{quadtrap} (b), and define the \emph{positive} and  \emph{negative} diagonals according to whether their slope is  positive or negative when the vertices of the trapezoid are horizontally shifted to transform the trapezoid into a rectangle. Then on a trapezoid, a trajectory passing from the left edge to the top edge, the left edge to the right edge, the bottom edge to the top edge, or the bottom edge to the right edge crosses the negative diagonal. Of these, only trajectories passing from the left edge to the right edge, or from the bottom edge to the top edge, cross the positive diagonal.
\end{corollary}

The key question, to determine the derived sequence from the original cutting sequence, is whether a trajectory that crosses edges $E_1 E_2 E_3$ crosses $E_2' = \vm(E_2)$. Lemmas \ref{hitisgood} and \ref{goodishit} establish that this happens if and only if the word $E_1 E_2 E_3$ is in a $(00)$ or $(11)$ case. Each Lemma proves one direction.

\begin{lemma} \label{hitisgood}
Assume that $E_2$ is the gluing edge for a typical cylinder. Consider the word $E_1 E_2 E_3$ and the corresponding portion of the trajectory that lies between the intersections with edges $E_1$ and $E_3$. Consider the shorter segment $s$ that is the intersection of that segment with the slanted trapezoid that has $E_2$ as a negative diagonal. If $s$ crosses $E_2'$, then $E_1 E_2 E_3$ is a $(00)$ or $(11)$ word.
\end{lemma}

\begin{proof}
Refer to  Figure \ref{hit}. The hypothesis states that $s$ crosses the shaded trapezoid's positive diagonal $B'$. By Corollary \ref{zerocor}, there are two ways for this to occur: Either $s$ passes from the left edge $C$ to the right edge $D$, or from the bottom edge $G$ to the top edge $F$. From this, we wish to recover the information about which of the gluing edges, which are the bold edges $I,K,A,B,L,N$, are hit by the extension of $s$ to a trajectory. This will determine which case the word is in.

\begin{figure}[!h] 
\centering
\includegraphics[width=180pt]{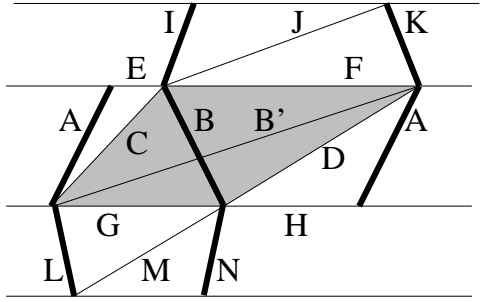} 
\begin{quote}\caption{ A typical cylinder with a slanted trapezoid shaded, and parts of two adjacent cylinders    \label{hit}} \end{quote}
\end{figure}
Suppose that $s$ passes from the left edge to the right edge of the shaded trapezoid, so from edge $C$ to $D$.  The angle restriction  $\theta \in [0,\ts)$ requires that the trajectory is less steep than all of the slanted edges. So a trajectory intersecting edge $C$ must hit edge $A$ just before it hits $C$, because the angle restriction prevents its coming from edge $L$. By Corollary \ref{zerocor}, a trajectory intersecting edge $C$ and then edge $D$ must intersect edge $B$ in between. Similarly, because of the angle restriction, a trajectory intersecting edge $D$ must then intersect edge $A$, because it is not steep enough to intersect edge $K$. So a trajectory passing from edge $C$ to $D$ must cut through gluing edges $ABA$. This is a $(00)$ case.

Suppose that $s$ passes from the bottom edge to the top edge of the shaded trapezoid, so from edge $G$ to edge $F$. Because the trajectory must be less steep than slanted edge $M$, it must intersect $L$ just before edge $G$. By Corollary \ref{zerocor}, between its intersections with edges $G$ and $F$, it must pass through edge $B$. Again, because the trajectory must be less steep than the slanted edge $J$, it must intersect $K$ just after $F$. So a trajectory passing from edge $G$ to $F$ must cut through gluing edges $LBK$. This is a $(11)$ case.
\end{proof}

\begin{lemma} \label{goodishit}
Assume that $E_2$ is the gluing edge for a typical cylinder. Consider the word $E_1 E_2 E_3$ and the corresponding portion of the trajectory that lies between the intersections with edges $E_1$ and $E_3$. Consider the shorter segment $s$ that is the intersection of that segment with the slanted trapezoid that has $E_2$ as a negative diagonal. If $E_1 E_2 E_3$ is a $(00)$ or $(11)$ word, then $s$ crosses $E_2'$.
\end{lemma}

\begin{proof}
Suppose that $E_1 E_2 E_3$ is a  $(00)$ word. Referring to Figure \ref{hit}, this means that the corresponding trajectory passes from edge $A$ to $B$ to $A$. In so doing, $s$  crosses edge $C$ and then $D$. This means that $s$ goes from the left edge of the slanted trapezoid to the right edge, so by Corollary \ref{zerocor}, $s$ intersects the positive diagonal $B'$ of the slanted trapezoid, which is $E_2'$ as desired.

Now suppose that $E_1 E_2 E_3$ is a $(11)$ word. In Figure \ref{hit}, this means that the corresponding trajectory passes from edge $L$ to $B$ to $K$. In so doing, $s$  crosses edge $G$ and then $F$. This means it goes from the bottom edge of the slanted trapezoid to the top edge, so by Corollary \ref{zerocor}, $s$ intersects the positive diagonal $B'$ of the slanted trapezoid, which is $E_2'$ as desired.
\end{proof}

Most of the cylinders in the surfaces we discuss in this paper are the union of two trapezoids, covered in Lemma \ref{traptrap}. Sometimes, as on each end of a \bm surface, one of the trapezoids is degenerate and is a triangle, and other times, as in a double $(2n\!+\!1)$-gon, both of the trapezoids in the cylinder decomposition degenerate to triangles. We consider these to be special cases of two trapezoids, so we cover these as part of Lemma \ref{traptrap}. Sometimes, as in a regular $4n$-gon, the middle cylinder is a rectangle, but a double cover (two rectangles) is two trapezoids, so we cover this in Lemma \ref{traptrap} as well.  Finally, horizontal edges are not the gluing edges for any cylinder, so we consider them in Lemma \ref{horiz}.

\begin{lemma} \label{traptrap}
Let $A$ be a gluing edge of a typical cylinder of a perfect translation surface of common modulus $M$. Consider a trajectory with $\theta \in [0,\ts)$ that crosses $A$, and let $s$ be the segment that is the intersection of the trajectory with the slanted trapezoid that has $A$ as its negative diagonal.  Then $s$ crosses $A' = V_M (A)$  if and only if the corresponding $A$ in the cutting sequence occurs as the middle letter of a $(11)$ or $(00)$ word. 
\end{lemma}


\begin{proof} 
This is a direct corollary of Lemmas \ref{hitisgood} and \ref{goodishit}. 

Because this is the key result of the paper, we will draw several pictures to carefully work through why this is true.

There are three cases of a typical cylinders: A cylinder that is the union of two isosceles trapezoids, and then the two limiting cases, where one of the trapezoids is an isosceles triangle, or where both trapezoids are isosceles triangles. We will treat the general case first, and then give further arguments for the limiting cases.



\begin{figure}[!h] 
\centering
\includegraphics[width=300pt]{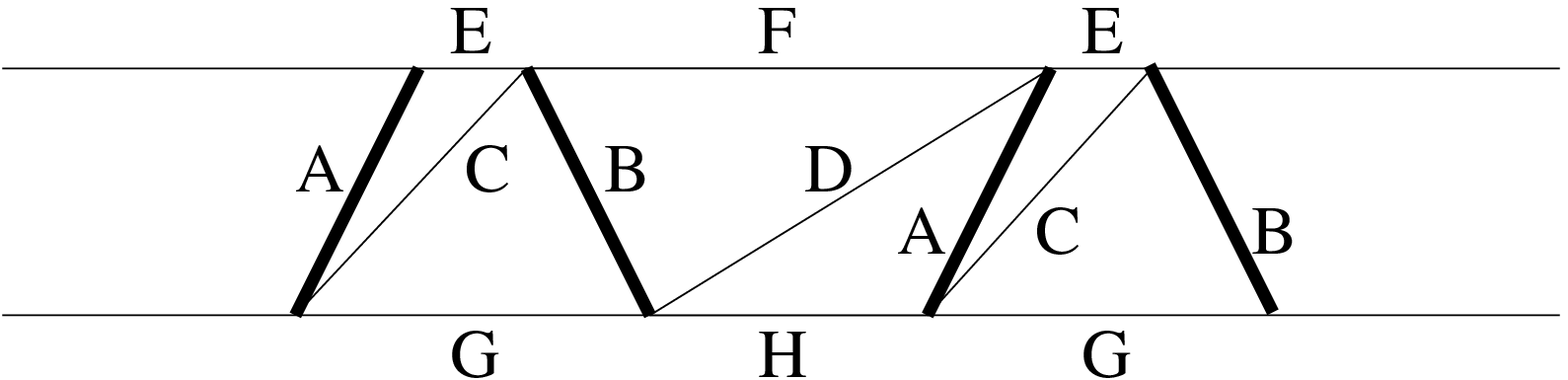} 
\includegraphics[width=300pt]{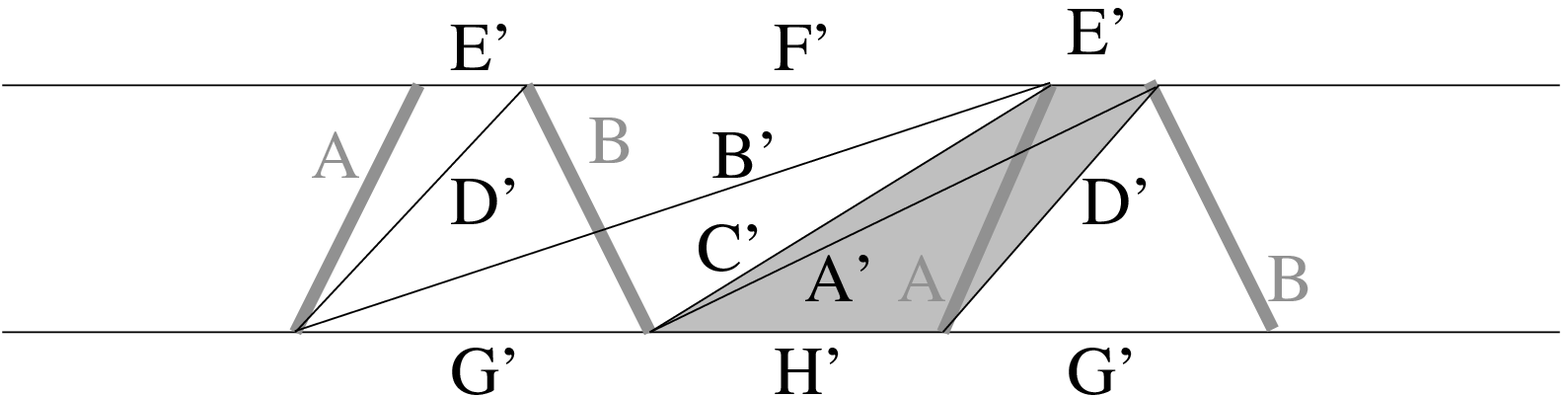} 
\begin{quote}\caption{ A typical cylinder, and the result of applying $V_M$, depicted in the universal cover. The slanted trapezoid with $A,A'$ as diagonals is shaded   \label{uca}} \end{quote}
\vspace{-2em}
\end{figure}

Consider a cylinder $\mathscr{C}$ of modulus $M$ that is the union of two isosceles trapezoids glued along edges $A$ and $B$. The top half of Figure \ref{uca} shows a portion of the universal cover of the cylinder $\mathscr{C}$. This is a copy of Figure \ref{unicov}. The bottom half of Figure \ref{uca} shows the action of $V_M$ on $\mathscr{C}$, again depicted in the universal cover. This is a copy of Figure \ref{ucshear} (b), except that we have shaded a different trapezoid.

In a perfect translation surface, each cylinder is partitioned into two isosceles trapezoids (in a typical case) or one rectangle (in an exceptional case), divided by the gluing edges. Now, we use a different partition, into slanted trapezoids, where the two trapezoids are divided by the slanted edges.

Consider a trajectory with $\theta \in [0,\ts )$ crossing the shaded slanted trapezoid in Figure \ref{uca}. By Corollary \ref{zerocor}, every such trajectory crosses $A$. By Lemmas $\ref{hitisgood}-\ref{goodishit}$, the trajectory crosses $A'$ if and only if it is a $(00)$ or $(11)$ case.

The same analysis applies to the other slanted trapezoid in the cylinder (shaded in Figure \ref{ucb}): Every   trajectory with $\theta \in [0,\ts )$ that crosses the shaded slanted trapezoid crosses $B$, but only those in the $(00)$ or $(11)$ cases cross $B'$.

\begin{figure}[!h] 
\centering
\includegraphics[width=300pt]{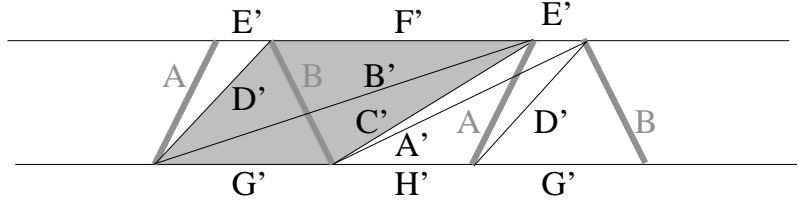} 
\begin{quote}\caption{   A typical cylinder, and the result of applying $V_M$, depicted in the universal cover. The slanted trapezoid with $B,B'$ as diagonals is shaded  \label{ucb}} \end{quote}
\end{figure}


To spell things out in detail, we will discuss the limiting cases where one or two of the trapezoids are triangles. 

If one of the trapezoids degenerates to a triangle, this means that edge $E'$ (without loss of generality) is contracted to a point, and edge $D'$ is not needed.  
Consider edge A.
The top edge has been contracted to a point, so only the cases $(00)$ and $(10)$ are possible. A trajectory in the $(10)$ case crosses edge $H'$ and then edge $A$. Such a trajectory clearly cannot cross edge $A'$. A trajectory in the $(00)$ case crosses $BAB$, so by Corollary \ref{zerocor}, it crosses $A'$. So indeed, the trajectory crosses $A'$ if and only if it is in a $(00)$ or $(11)$ case, because the $(11)$ case is not possible. For edge $B$,
the analysis is the same as in the two-trapezoid case.

If both of the trapezoids degenerate to triangles, this means that both edges $E'$ and $H'$ contract to a point. Edges $A,A',C',D'$ all coincide, so  edges $C'$ and $D'$ are  not needed and edge $A'$ coincides with edge $A$.
Thus, trivially, a trajectory crosses $A'$ if and only if it crosses $A$. The only way to cross edge $A$ is in the $(00)$ word $BAB$, so again, the trajectory crosses $A'$ if and only if it is in a $(00)$ or $(11)$ case.
For edge $B$ (looking at the slanted trapezoid whose diagonals are $B$ and $B'$), the analysis is the same as in the two-trapezoid case.

Finally, we will consider an exceptional cylinder. We take a double cover of the exceptional cylinder, so in Figure \ref{ucb} this means that gluing edges $A$ and $B$ have the same label (say, $A$) and are vertical. Similarly, each pair of edge labels $C$ and $D$,  $E$ and $F$, and $G$ and $H$ is given the first label ($C,E,G$). We consider the slanted trapezoid $CE'CG'$, with positive diagonal $A'$ and negative diagonal $A$, and then the analysis is the same as in the typical case.

For an expanded discussion of the degenerate and exceptional cases, see (\cite{thesis}, Chapter 3 and Lemma 4.8).

\end{proof}

\newpage
\begin{lemma} \label{horiz}
A trajectory crosses a horizontal edge $A$ if and only if it crosses its image $A' =V_M (A)$.
\end{lemma}

\begin{proof}
The horizontal shear $S_M$ is an isometry on horizontal segments, and $R$ is a reflection, so $V_M = S_M \circ R$ is an isometry on horizontal segments. This means that $A=A'$, so the result is trivial.
\end{proof}

\begin{lemma} \label{inorder}
The derived sequence consists of the edge labels of the ``primed'' edges crossed by the trajectory, in the same order.
\end{lemma}

\begin{proof}
We partitioned the infinite trajectory into countably many finite segments, each of which is the intersection of the trajectory with a slanted trapezoid. Lemma \ref{lookatshears} tells us that the derived sequence consists of the crossed edge labels. These are a subset of the full sequence of edge labels, in the same order, so the order does not change.   
\end{proof}

Now we may prove the theorem:

\begin{proof}[Proof of Theorem \ref{mainthm}]
By Lemma \ref{lookatshears}, determining the derived cutting sequence for a given trajectory is equivalent to determining which sheared edges the trajectory crosses. Lemmas \ref{traptrap}$-$\ref{horiz} prove that the trajectory crosses the sheared edges as stated in the theorem. Lemma \ref{inorder} tells us that their order does not change. This proves the result.
\end{proof}

\section{Regular Polygon Surfaces}
\label{regularpolygons} 

We will prove a corollary of Theorem \ref{mainthm}, that the combinatorial  rule for obtaining the derived sequence from the original cutting sequence for regular polygon surfaces is to keep only the \emph{sandwiched} edge labels:

\begin{definition} \label{sandwiched}
A \emph{sandwiched} label is one that is preceded and followed by the same label, as $A$ is in the sequence $\ldots CBABDAC\ldots$.
\end{definition}

This result was already known for even polygon surfaces \cite{SU} and double odd polygon surfaces \cite{Davis} but this is a new method of proof.

\begin{definition} \label{regdef}
A \emph{regular $n$-gon surface} is a double regular $n$-gon for $n$ odd, and may be a single or double regular $n$-gon for $n$ even. For a single $n$-gon,  identify  opposite parallel edges  to obtain a translation surface (this is only possible when $n$ is even). For a double $n$-gon, take two copies of the regular $n$-gon, one of which is a reflection of the other, and identify opposite parallel edges. 
\end{definition}

John Smillie and Corinna Ulcigrai analyzed the even cases and showed, among many other results for these surfaces, that the combinatorial  rule for obtaining the derived sequence from the original cutting sequence for regular $2n$-gon surfaces is to keep only the sandwiched edge labels (\cite{SU}, Proposition 4.1.1). They used ``transition diagrams,'' which we explore further in Section \ref{td}. We subsequently showed that Smillie and Ulcigrai's transition diagram method extends to double  odd regular polygons as well (\cite{Davis}, Theorem 6.4). Note that Smillie and Ulcigrai proved many results for the regular octagon surface, and for even polygon surfaces in general, and we only extend a few of their results to other surfaces.

The two rules in Theorem \ref{mainthm} are reduced to this one simple rule (``keep only the sandwiched edge labels'') in the case of regular polygon surfaces. The key insight is that because of the surfaces' symmetry, sandwiching occurs exactly in the $(00)$ and $(11)$ cases. 

In  Lemmas \ref{mododd} (double) and \ref{modeven} (single), we prove that regular $n$-gon surfaces are perfect, with common modulus $2\cot\pi/n$, so that we may apply the Main Theorem \ref{mainthm} to them. 
We will need the following trigonometric identities:

\begin{eqnarray}
1+2\cos\theta+\ldots+2\cos k\theta + \cos (k+1) \theta &=& 2\cot (\theta/2) \sin((k+1)\theta); \label{cotone} \\
1+2\cos\theta+\ldots+2\cos k\theta &=& \frac{\sin((k+1/2)\theta)}{\sin (\theta/2)}. \label{dirichlet}
\end{eqnarray}

Equation (\ref{cotone}) is proven in (\cite{Davis}, Lemma 2.1). To prove Equation (\ref{dirichlet}), write the sum as $\sum_{j=-k}^{k} \cos j\theta$ and then rewrite $\cos j\theta = \frac{e^{ij\theta}+e^{-ij\theta}}{2}$ and simplify the expression. For details, see \cite{thesis}.



\begin{lemma} \label{mododd}
A double regular $n$-gon surface is perfect, with common modulus $2\cot \pi/n$.
\end{lemma}

\begin{proof}
We will prove that each cylinder of a double regular $n$-gon surface (a) is the union of two isosceles trapezoids or two isosceles triangles, and (b) has modulus $2\cot \pi/n$.

(a) The interior of the polygon that lies between horizontal levels $k$ and $k+1$ is an isosceles trapezoid for $k=0,\ldots,n-2$ and is an isosceles triangle for $k=n-1$ when $n$ is odd. A cylinder consists of gluing together two congruent copies of such pieces, one from each polygon.

(b) Assume that the edges of a regular polygon have length $1$. The exterior angle of the polygon is $\theta = 2\pi/n$. In a regular polygon, the horizontal segment connecting the vertices at level $0$ is $1$ (the horizontal edge), at level $1$ is $1+2\cos\theta$, and at level $k\neq 0$ in general is \mbox{$1+2(\cos\theta+\ldots+\cos(k\theta))$} (see Figure \ref{regwidth}). 

\begin{figure}[!h] 
\centering
\includegraphics[width=380pt]{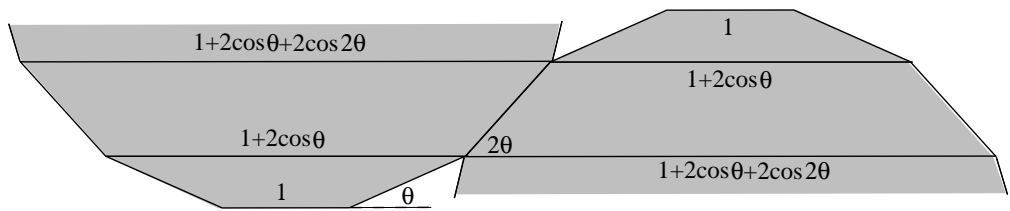} 
\begin{quote}\caption{Two regular polygons, with opposite parallel edges \mbox{identified}  \label{regwidth}} \end{quote}
\end{figure}

When we glue two regular polygons together to make a surface, two copies of the trapezoid or triangle between level $k$ and level $k+1$ are glued together to make the $k$th cylinder. So the width of the $k$th cylinder is 
\begin{eqnarray*}
1+2(\cos\theta+\ldots+\cos(k\theta))\hspace{1em}+\hspace{1em}1+2(\cos\theta+\ldots+\cos((k+1)\theta)) \\
= 2\left(1+2\cos\theta + \ldots + 2\cos k\theta + \cos (k+1)\theta\right)=2\cot(\theta/2) \sin ((k+1)\theta),
\end{eqnarray*}
by applying Equation (\ref{cotone}).

The vertical distance between level $k$ and level $k+1$ is $\sin (k+1)\theta$.

The modulus of the $k$th cylinder is then 
$$\frac{2\cot(\theta/2) \sin (k+1)\theta}{\sin (k+1) \theta} = 2\cot \theta/2 =  2\cot \pi/n.$$\end{proof}

\begin{lemma} \label{modeven}
A regular $2n$-gon surface is perfect, with common modulus $2\cot \pi/n$.
\end{lemma}

\begin{proof} We will prove that each cylinder of a single regular $2n$-gon surface is the union of two isosceles trapezoids (typical), or is one rectangle (exceptional). In the typical case, the cylinder has modulus $2\cot \pi/n$. In the exceptional case, the cylinder has modulus $\cot \pi/n$. 

The region that lies between a pair of  non-horizontal or non-vertical edges of an even-gon is an isosceles trapezoid. It is not a triangle, because $2n$-gons have top and bottom horizontal edges. For the cylinders that are the union of two trapezoids, the calculation is the same as in Lemma \ref{mododd}, so the modulus is $2\cot\pi/n$. 

A $4m$-gon with horizontal edges also has two vertical edges, which are identified. The cylinder glued along the vertical edge consists of a single rectangle. The rectangle's horizontal edges are between (using the same notation as in Lemma \ref{mododd}) levels $m-2$ and $m-1$, so the width is \mbox{$1+2(\cos\theta+\ldots+\cos m\theta)$} and its height is $1$, so its modulus is just the width:
$$1+2\cos\theta+\ldots+2\cos(m\theta)=\frac{\sin((m+1/2)\theta)}{\sin(\theta/2)},$$
by applying Equation (\ref{dirichlet}). 

Applying the angle-addition formula, we can rewrite this as
$$\frac{\sin (m\theta) \cos (\theta/2) + \cos (m\theta)\sin (\theta/2)}{\sin (\theta/2)}=\cot \theta/2=\cot\pi/n,$$
since $m\theta = m(2\pi/(4m)) = \pi/2$, so $\sin (m\theta)=1,\cos(m\theta)=0$. 
This is a single rectangle of modulus $\frac 12 (2\cot\pi/n)$, so the surface is perfect.
\end{proof}

\begin{lemma}\label{euclid}
For a regular $n$-gon surface, $\ts = \pi/n$.
\end{lemma}

\begin{proof}
A regular polygon is cyclic.
The angle of each slanted edge can be taken as the angle between two of the polygon's diagonals. Each of these pairs of diagonals spans a chord of the circle.
The Central Angle Theorem says that this angle is half that of the central angle spanned by the same chord. The result follows.
\end{proof}



\begin{theorem} \label{regularsandwich}
Consider a regular $n$-gon surface, and let $M=2\cot\pi/n$. Consider a geodesic trajectory on the surface whose angle $\theta$ satisfies $\theta \in [0,\pi/n)$, and its associated cutting sequence. Then the derived sequence consists precisely of the sandwiched letters, read from left to right. 
\end{theorem}

To streamline the proof, we introduce the following terminology:

\begin{definition}
If a given letter in the cutting sequence survives into the derived sequence, we say that it is \emph{kept}.
\end{definition}

Theorem \ref{mainthm} says that an gluing edge label is kept if and only if it is the middle letter of a $(00)$ or $(11)$ word. We must show that this happens for regular polygon surfaces if and only if the edge label is sandwiched. We will show this for gluing edges and for horizontal edges. This will prove the result.

\begin{proof}
By Lemmas \ref{mododd} and \ref{modeven}, regular polygon surfaces are perfect with common modulus $M$, and by Lemma \ref{euclid}, $\ts = \pi/n$. So the hypotheses of Theorem \ref{mainthm} are satisfied, and we may apply it. 

Suppose that a given gluing edge label is sandwiched. Then it occurs in the middle of a three-label sequence $E_1 E_2 E_1$ for some edge labels $E_1,E_2$ (not necessarily distinct).  If $E_1$ and $E_2$ are at the same level, then the transition $E_1 E_2$ is type $(0)$, and the transition $E_2 E_1$ is also case $(0)$, so it is case $(00)$ and by Theorem \ref{mainthm}, $E_2$ is kept. Similarly, if $E_1$ and $E_2$ are on different levels, then both transitions are type $(1)$, so it is case $(11)$ and by Theorem \ref{mainthm}, $E_2$ is kept.

Suppose that a given gluing edge label is not sandwiched. Then it occurs in the middle of a three-label sequence $E_1 E_2 E_3$ for some edges $E_1,E_2, E_3$ not necessarily all distinct, but $E_1\neq E_3$. Because regular polygon surfaces are perfect, each cylinder has at most two gluing edges, so at least two of the edges must be in different levels. If $E_1,E_3$ are in the same level and $E_2$ is in a different level, then $E_1=E_3$ by symmetry, so $E_2$ is sandwiched, a contradiction. If $E_1,E_2$ are in the same level and $E_3$ is in a different level, or if $E_2,E_3$ are in the same level and $E_1$ is in a different level, these are cases $(01)$ and $(10)$, respectively, so by Theorem \ref{mainthm}, $E_2$ is not kept.

Theorem \ref{mainthm} says that a horizontal edge label is always kept. By symmetry, a horizontal edge label is always sandwiched in regular polygon surfaces. So a horizontal edge label is kept if and only if it is sandwiched.
\end{proof}

\section{Introduction to \mbox{Bouw-M$\mathbf{\ddot{o}}$ller Surfaces}}
\label{bmintro} 
In Lemmas \ref{mododd} and \ref{modeven}, we showed that for the $n$-indexed family of regular $n$-gon surfaces, the cylinder moduli are equal, or in a $2:1$ ratio. It turns out that there is an $(m,n)$-indexed family of surfaces whose cylinder moduli are all equal, called \emph{\bm surfaces}.

As discussed in the introduction, after Irene Bouw and Martin M\"oller gave algebraic models for their \bm surfaces \cite{BM}, Pat Hooper gave a polygon decomposition for them \cite{Hooper}, which we will use here. The polygon decomposition gives a family of translation surfaces created by identifying opposite parallel edges of a collection of $m$ ``semi-regular'' polygons that each have $2n$ edges. 

In Definitions \ref{semiregular} and \ref{bmdefinition}, we  define \bm surfaces. In Lemmas \ref{bmlevel} and \ref{modbm}, we  show that they are perfect translation surfaces, so we can apply the main Theorem \ref{mainthm}, which we do in Chapter \ref{bmresults}. 

The language of Definitions \ref{semiregular} and \ref{bmdefinition} is taken directly from \cite{Hooper}.

\begin{definition} \label{semiregular}
The $(a,b)$ semi-regular $2n$-gon has edge vectors given by:
\[
\mathbf{v}_i = 
\begin{cases}
 a \ [\cos \frac{i\pi}{n},\sin  \frac{i\pi}{n}] & \text{if }i\text{ is even} \\
 b\ [\cos \frac{i\pi}{n}, \sin  \frac{i\pi}{n}] & \text{if }i\text{ is odd}
\end{cases}
\]
for $i=0,\ldots,2n-1$. Denote this $2n$-gon by $P_n (a,b)$. The edges whose edge vectors are $\mathbf{v_i}$ for $i$ even are called \emph{even edges}. The remaining edges are called \emph{odd edges}. We restrict to the case where at least one of $a$ or $b$ is nonzero. If $a$ or $b$ is zero, $P_n (a,b)$ degenerates to a regular $n$-gon.
\end{definition}


\begin{lemma} \label{bmlevel} \label{srspecial}
Semi-regular polygons are special.
\end{lemma}

\begin{proof}
Recall Definition \ref{specialdef}, that a polygon is \emph{special} if it is convex, level and vertically symmetric.

Semi-regular polygons are convex because the exterior angle increases monotonically from $0$ to $(2n-1)/n$.

To show that they are level, we must show that vector $\mathbf{v}_i$ has the same $y$-value as vector $\mathbf{v}_{2n-i}$. To show that they are vertically symmetric, we will show that the region of a semi-regular polygon between two levels is an isosceles trapezoid or triangle. To do this, we will show that the $x$-values of $\mathbf{v}_i$ and $\mathbf{v}_{2n-i}$  have the same magnitude. If $i$ is even, 
\begin{eqnarray*}
\mathbf{v}_{2n-i} = a  \left [ \cos \frac{(2n-i)\pi}{n},\sin  \frac{(2n-i)\pi}{n}\right ]   
&=& a \left [ -\cos \left(  \frac{i\pi}{n} \right),\sin \left(\frac{i\pi}{n} \right)\right]. 
\end{eqnarray*}
Since $\mathbf{v}_{i} = a \left [ \cos \left(  \frac{i\pi}{n} \right),\sin \left(\frac{i\pi}{n} \right)\right]$, this proves both results.

For $i$ odd, $a$ is replaced with $b$.
\end{proof}

\begin{corollary}
Each \bm surface admits the Veech element $R$, whose derivative is a vertical flip. This follows from the vertical symmetry.
\end{corollary}

\begin{definition} \label{bmdefinition}
The $(m,n)$ \bm surface is made by identifying the edges of $m$ semi-regular polygons $P(0),\ldots,P(m-1)$. Define $P(k)$ by

\[
P(k) = 
\begin{cases}
 P_n \left(\sin \frac{(k+1)\pi}{m}, \sin \frac{k\pi}{m}\right) & \text{if }m\text{ is odd} \\
P_n \left(\sin \frac{k\pi}{m}, \sin \frac{(k+1)\pi}{m}\right) & \text{if }m \text{ is even and } k \text{ is even} \\
P_n \left(\sin \frac{(k+1)\pi}{m}, \sin \frac{k\pi}{m}\right) & \text{if }m \text{ is even and } k \text{ is odd.} \\
 \end{cases}
\]
We form a surface by identifying the edges of the polygon in pairs. For $k$ odd, we identify the even edges of $P(k)$ with the opposite edge of $P(k+1)$, and identify the odd edges of $P(k)$ with the opposite edge of $P(k-1)$. The cases in the definition of $P(k)$ are chosen so that this gluing makes sense. 
\end{definition}


The $(m,n)$ \bm surface is a collection of $m$ polygons, where the first and $m$th are regular $n$-gons, and the $m-2$ middle polygons are semi-regular $2n$-gons (see Figure \ref{bmsurface}). If $m$ is odd, the central polygon is regular, 
 and if $m$ is even, the two central polygons are regular.

\begin{figure}[!h] 
\centering
\includegraphics[width=400pt]{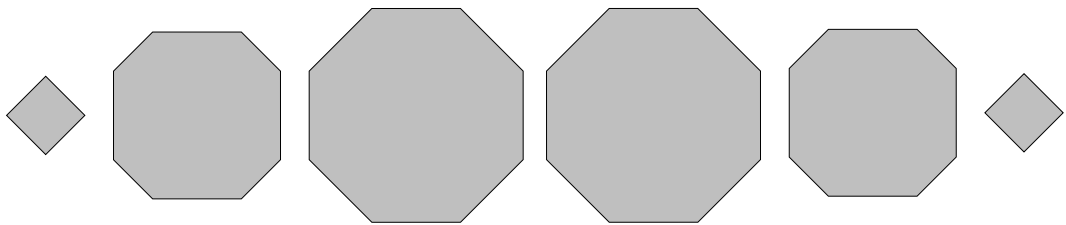}
\begin{quote}\caption{The \bm surface $M_{(6,4)}$. Edges of the square $P(0)$ on the far left are glued to oppositely-oriented parallel edges of the semi-regular octagon $P(1)$. The remaining edges of $P(1)$ are glued to the regular octagon $P(2)$, and so on. 
\label{bmsurface}} \end{quote}
\end{figure}

\bm surfaces are translation surfaces because each surface is a collection of polygons with opposite parallel edges identified. Each $2n$-gon has $n$ of its edges glued to the polygon on its left, and the other $n$ edges glued to the polygon on its right. Each of the $n$-gons on the ends is only glued to one other polygon.

\begin{corollary}[to Lemma \ref{srspecial}] \label{bmspecial} \bm surfaces are special. 
\end{corollary}

\begin{proof}
Lemma \ref{srspecial} shows that semi-regular polygons are special, and Definition \ref{bmdefinition} gives a decomposition of a \bm surface into semi-regular polygons.
\end{proof}

\begin{lemma} \label{modbm}
\bm surfaces are perfect.
\end{lemma}

\begin{proof}
Recall (Definition \ref{pts}) that a perfect translation surface is a special translation surface that has a perfect cylinder decomposition with some common modulus $M$. Corollary \ref{bmspecial} shows that \bm surfaces are special. Showing that \bm surfaces have a perfect decomposition requires that every cylinder (a) is the union of two isosceles trapezoids, and (b) has the same modulus. 

(a) We showed in Lemma \ref{bmspecial} that \bm surfaces are special, so the only way a cylinder could fail to be the union of two isosceles trapezoids is if it consisted of a single rectangle with just one gluing edge. This occurs when a vertical edge of a polygon is identified with an opposite parallel edge of the same polygon. By the gluing rules in Definition \ref{bmdefinition}, each polygon edge is glued to an edge of another polygon, not an edge of the same polygon, so this cannot occur, and every cylinder is the union of two isosceles trapezoids.

 (b) In fact, the edge lengths of \bm surfaces are chosen precisely so that the modulus is the same; see  \cite{BM}, Theorem $8.1$. The calculation is a elementary but lengthy trigonometric result: First, we calculate the length of the horizontal segment connecting the two vertices of polygon $P(k)$ at each level, and then add two of these together to find the width of each cylinder. A simple calculation gives the height of the cylinder, and then the modulus of the cylinder is the ratio of width to height. It turns out that the modulus of every cylinder is \mbox{$2\cot\pi/n +2 \frac{\cos \pi/m}{\sin \pi/n}$}. For details, see (\cite{thesis}, Lemma 6.6).

\end{proof}

\section{Results for  Bouw-M$\mathbf{\ddot{o}}$ller Surfaces}
 \label{bmresults} 
 In Lemma \ref{modbm}, we  showed that \bm surfaces are are perfect. This tells us that we may apply the main Theorem \ref{mainthm} to them, which is what we do in this chapter. In $\mathsection$ \ref{bmcs}, we prove Theorem \ref{bmthm}, which describes how to determine the derived sequence from a cutting sequence on a \bm surface, and we  work through an example. In $\mathsection$ \ref{td}, we describe the \emph{transition diagrams} for \bm surfaces, extending some of John Smillie and Corinna Ulcigrai's ideas for regular polygons (\cite{SU}, $\mathsection$ 2.1.2) to these surfaces.

\subsection{Cutting sequences on \bm surfaces} \label{bmcs}

\begin{theorem} \label{bmthm}
Let \mbox{$M=2\cot\pi/n +2 \frac{\cos \pi/m}{\sin \pi/n}$}, the modulus of the cylinders of the $(m,n)$ \bm surface $S$. Consider a geodesic trajectory  with $0 \leq \theta < \ts$, and its associated cutting sequence. Mark a letter in the  cutting sequence if: 
\begin{itemize}
\item It corresponds to a gluing edge, and it is the middle letter of a $(00)$ or $(11)$ word, or
\item It corresponds to a horizontal edge.
\end{itemize}
Then the derived sequence consists precisely of the marked letters, read from left to right.
\end{theorem}

\begin{proof}
By Lemma \ref{modbm}, \bm surfaces are are perfect with common modulus $M$, so we may apply Theorem \ref{mainthm} to them, and the same result holds.
\end{proof}


\begin{example}
We will determine the derivation rule for cutting sequences on the $(3,4)$ \bm surface, shown in Figure \ref{34numbered}. (The ordering of the edge numbers may seem arbitrary, but it is chosen to simplify the \emph{transition diagrams} in Section \ref{td}.)

\begin{figure}[!h] 
\centering
\includegraphics[width=190pt]{hcd.eps} \hspace{1cm}
\includegraphics[width=190pt]{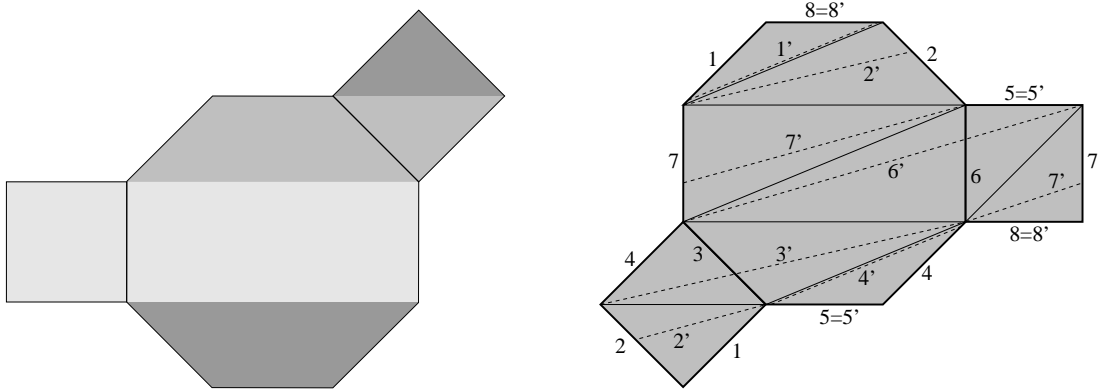}
\begin{quote}\caption{The $(3,4)$ \bm surface (a) showing the three cylinders, and (b) with edge labels, sheared edges (dotted lines) and auxiliary edges (thin) \label{34numbered}} \end{quote}
\end{figure}

We will consider the gluing edges one cylinder at a time, and then the horizontal edges.

First, we consider the cylinder with gluing edges $1$ and $2$. By Theorem \ref{bmthm}, $1$ and $2$ are kept if and only if they are the middle numbers of a $(00)$ or $(11)$ word. 

The $(00)$ word with $1$ as the middle number is $212$. The  $(00)$ and $(11)$ cases with $2$ as the middle number are  $121$ and $723$. The middle numbers are removed in the remaining $(01)$ and $(10)$ cases, which are $218$, $123$ and $721$.

By the same reasoning on the cylinder with gluing edges $3$ and $4$, $3$ is kept in the $(00)$ and $(11)$ cases $434$ and $236$, and removed in the other $(01)$ and $(10)$ cases $436$ and $234$. $4$ is kept in the $(00)$ case $343$ and removed in the $(10)$ case $543$. 

The gluing edges of the central cylinder are $6$ and $7$. The  $(00)$ cases are $767$ and $676$. The $(11)$ cases are $365$ and $872$. By Theorem \ref{bmthm}, the middle numbers are kept in these cases. The  $(01)$ cases are $765$ and $672$. The $(10)$ cases are $367$ and $876$. By Theorem \ref{bmthm}, the middle numbers are removed in these cases.

By Theorem \ref{bmthm}, horizontal edges $5$ and $8$ are always kept. Possible three-number transitions with middle numbers $5$ or $8$ are $654$ and $187$, so the middle number is kept in these cases.

The numbers that are kept and removed are summarized in Table \ref{34kept}. Of the $12$ three-letter transitions where the middle letter is kept, in the six $(00)$ cases the middle letter is sandwiched, and in the six $(11)$ cases, it is not.

\begin{table}[!ht]
\begin{center}
\begin{tabular}{|c|c|c|}
\hline
kept $(00)$ & kept $(11)$ & removed $(01),(10)$\\
\hline
212 & ~ & 218   \\
121 & 723 & 123, 721 \\
434 & 236 & 436, 234 \\
343 & ~ & 543  ~ \\
 & 654 & ~  ~ \\
767 & 365 & 765, 367 \\
676 & 872 & 672, 876 \\
~ & 187 & ~  ~ \\
\hline
\end{tabular}
\end{center}
\caption{All possible three-number transitions for the $(3,4)$ \bm surface, indicating when the middle number is kept. The table is sorted by middle number}
\label{34kept}
\end{table}

\end{example}

\subsection{Transition diagrams for \bm surfaces} \label{td}

To determine the derived sequence from a cutting sequences on a regular $2n$-gon surface, John Smillie and Corinna Ulcigrai used \emph{transition diagrams} (\cite{SU}, \mbox{$\mathsection$ 2.1.2}). In this section, we will describe the form of transition diagrams for \bm surfaces. Using the transition diagrams may yield a new proof for Theorem \ref{bmthm}, more similar to the proofs used in \cite{Davis}, \cite{SU} and \cite{SU2}.

\begin{definition}
A \emph{transition diagram} for a translation surface $S$ is a directed graph whose vertices are edge labels, and whose arrows connect two edge labels if a trajectory with $\theta \in [0,\ts)$ can intersect the first edge and then the second edge.
\end{definition}

\begin{example} \label{tdoctex}
We construct the transition diagram for the regular octagon surface (Figure \ref{tdoct}). For the regular octagon, $\ts = \pi/8$, so we consider trajectories with $\theta \in [0,\pi/8)$. A trajectory crossing horizontal edge $1$ must next cross edge $2$, so we connect $1$ to $2$ with an arrow. A trajectory crossing edge $2$ may next cross edge $1$ or edge $3$, so we connect $2$ to $1$ and to $3$. The rest of the transition diagram is constructed similarly.
\end{example}

\begin{figure}[!h] 
\centering
\includegraphics[width=400pt]{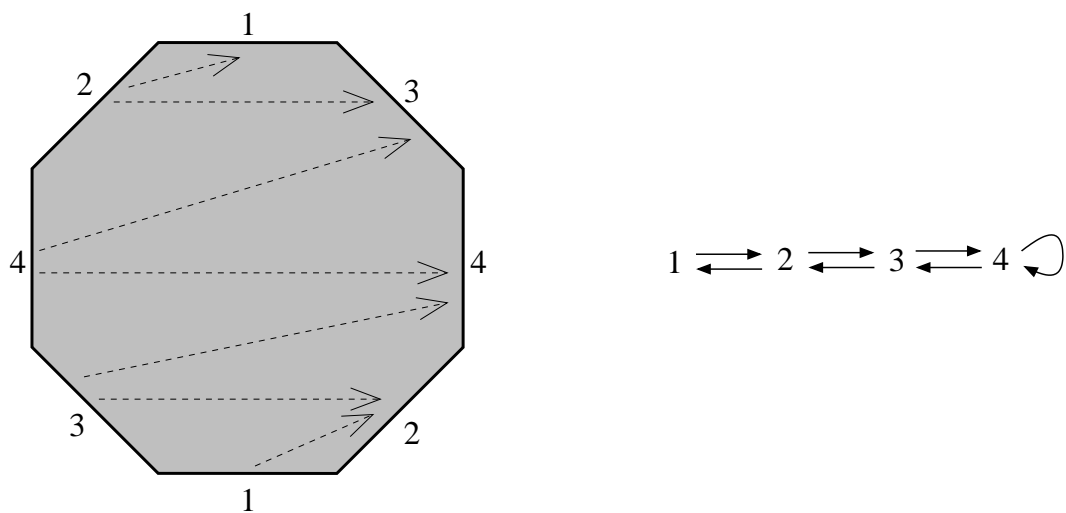} 
\begin{quote}\caption{The regular octagon surface, and its transition diagram \label{tdoct}} \end{quote}
\end{figure}

\begin{example}
We construct the transition diagram for the $(3,4)$ \bm surface (Figure \ref{td34}). For this surface, $\ts = \pi/8$, so we again consider trajectories with $\theta \in [0,\pi/8)$. A trajectory crossing edge $1$ then crosses edge $2$ or $8$, so we connect these with arrows. A trajectory crossing edge $2$ then crosses edge $1$ or $3$, so we connect these with arrows. The rest of the diagram is constructed similarly. 

Arranging the edge labels in numerical order in a straight line, as we did in Example \ref{tdoctex}, gives us the transition diagram shown at the top right of Figure \ref{td34}. We can rearrange the transition diagram into a rectangle, as shown at the bottom right of Figure \ref{td34}. 
\end{example}

\begin{figure}[!h] 
\centering
\includegraphics[width=400pt]{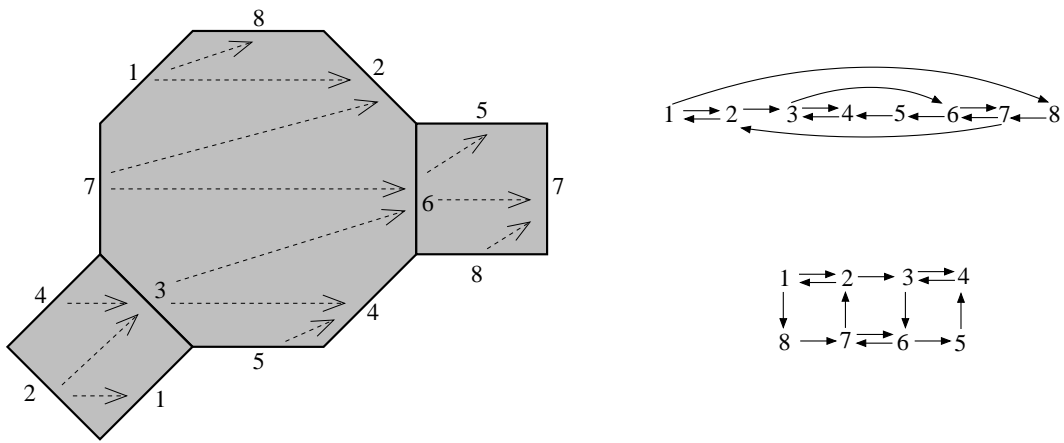} \hspace{1in}
\begin{quote}\caption{The $(3,4)$ \bm surface, and its transition diagram in two forms \label{td34}} \end{quote}
\vspace{-2.0em}
\end{figure}

It turns out that the rectangular form of the transition diagram in Figure \ref{td34} extends, and holds for \bm surfaces in general. Figure \ref{m6n5} shows the polygon decomposition and the transition diagram for the $m=6,n=5$ \bm surface. 

\begin{figure}[!h] 
\centering
\includegraphics[width=400pt]{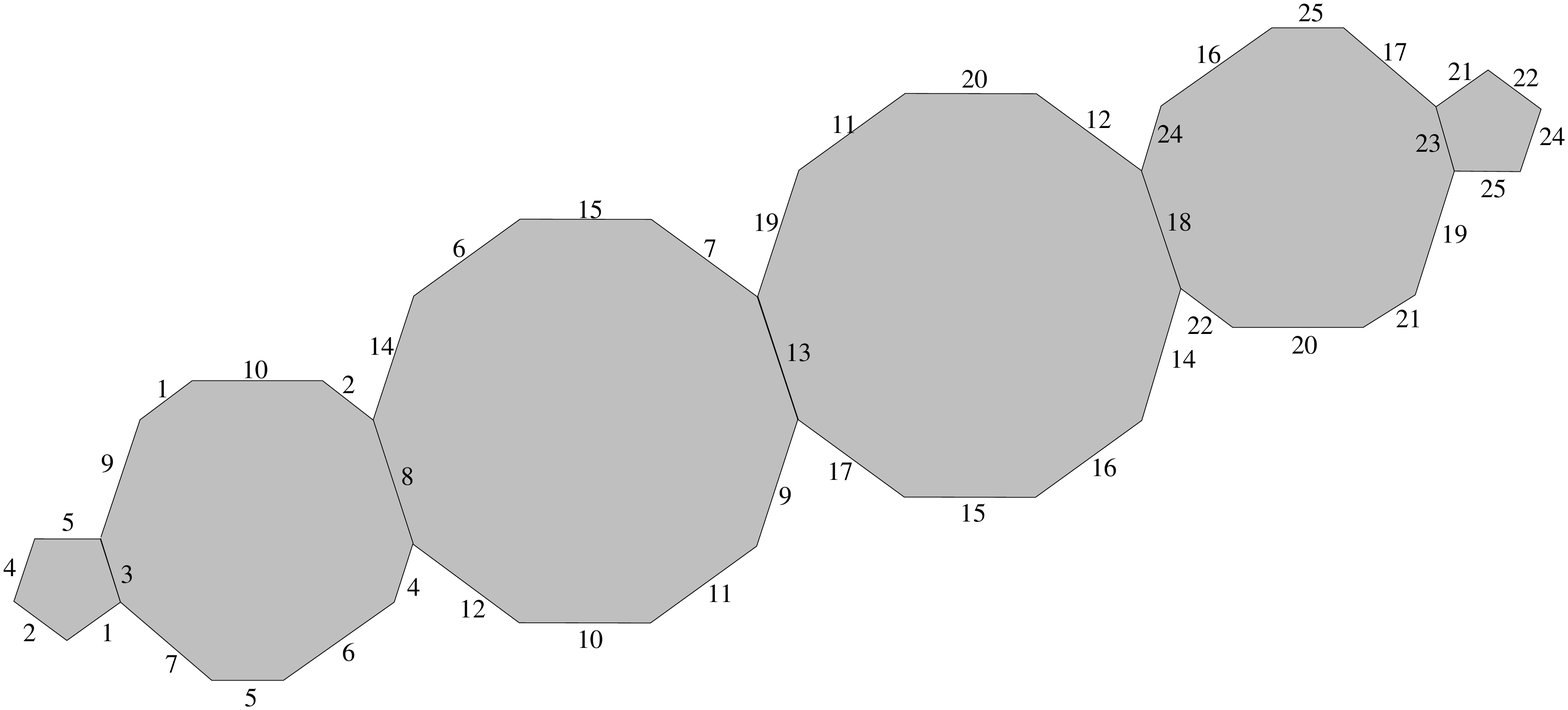}
\includegraphics[width=120pt]{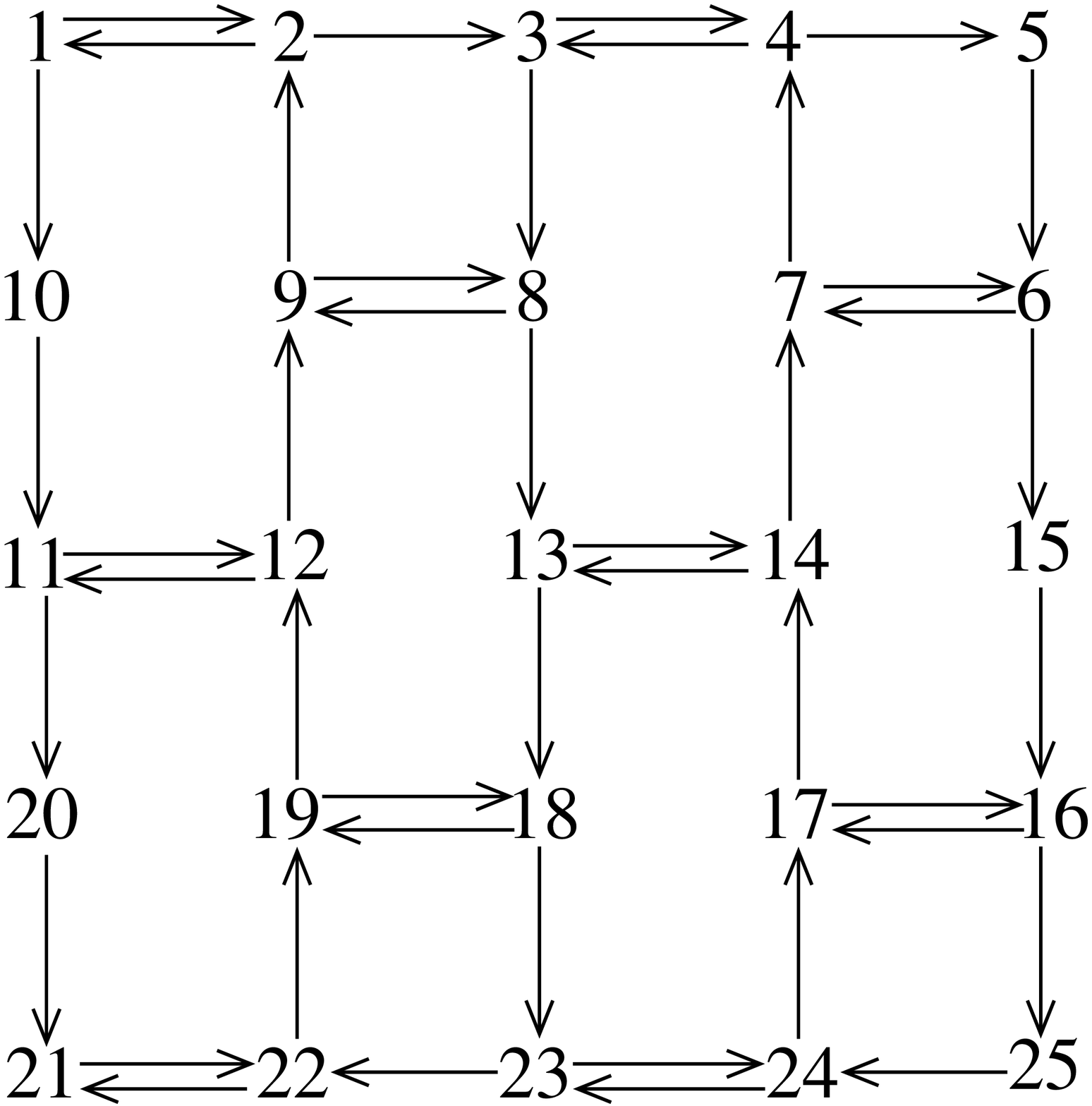} 
\begin{quote}\caption{The $(6,5)$ \bm surface and its transition diagram \label{m6n5}} \end{quote}
\vspace{-2.0em}
\end{figure}

There are $n$ numbers on each row, snaking back and forth, with $m-1$ rows in all. The edge labels in the $k$th row correspond to the edge labels that appear for the first time in $P(k-1)$. For example, in the $(3,4)$ \bm surface, $1,2,3,4$ appear in the first tilted square $P(0)$, $5,6,7,8$ appear in the octagon $P(1)$, and no new edge labels appear in the final square $P(2)$. 

The transition diagram for the $(m,n)$ \bm surface, with $m\geq 2, n\geq 3$, always follows the pattern of Figure \ref{m6n5}, with the upper-left corner the same and the diagram expanding (or contracting) down and to the right.
Showing that the transition diagrams always have this form is  tedious, so
we omit the proof.



\end{document}